\makeatletter
\newcommand{\xMapsto}[2][]{{\ext@arrow}0599{\Mapstofill@}{#1}{#2}}
\def\Mapstofill@{\arrowfill@{\Mapstochar\Relbar}\Relbar\Rightarrow}
\makeatother

\documentclass[oneside,11pt]{amsart}
\usepackage[margin=3cm,marginpar=2.5cm]{geometry}
\usepackage[utf8]{inputenc}
\usepackage{etoolbox}

\usepackage{centernot}
\usepackage{mathtools}
\usepackage{stmaryrd}

\usepackage{amsmath}
\usepackage{amssymb}
\usepackage{amsfonts}
\usepackage{soul}
\usepackage{amsthm}
\usepackage{graphicx}
\usepackage{tikz,tikz-cd}
\usepackage{color}
\usepackage{mathtools}
\usepackage{enumitem}
\usepackage{hyperref}

\newtheorem{thm}{Theorem}[section]
\newtheorem{prop}[thm]{Proposition}
\newtheorem{lem}[thm]{Lemma}
\newtheorem{cor}[thm]{Corollary}

\theoremstyle{definition}
\newtheorem{defn}[thm]{Definition}
\newtheorem{rem}[thm]{Remark}
\newtheorem{conv}[thm]{Convention}
\newtheorem{exmp}[thm]{Example}
\newtheorem{ques}[thm]{Question}

\renewcommand{\bar}[1]{\overline{#1}}

\newcommand{\set}[2]{\{\,{#1} \mid{#2} \,\}}

\renewcommand{\emptyset}{\varnothing}

\renewcommand{\setminus}{-}

\newcommand{\field}[1]{\mathbb{#1}}
\newcommand{\Z}{\field{Z}}
\newcommand{\R}{\field{R}}

\newcommand{\N}{\field{N}}
\newcommand{\E}{\field{E}}

\newcommand{\PP}{\field{P}}

\newcommand{\cH}{\mathcal{H}}
\newcommand{\cG}{\mathcal{G}}
\newcommand{\cV}{\mathcal{V}}
\newcommand{\cL}{\mathcal{L}}
\newcommand{\cE}{\mathcal{E}}

\newcommand{\cA}{\mathcal{A}}
\newcommand{\cP}{\mathcal{P}}
\newcommand{\cS}{\mathcal{S}}
\newcommand{\cC}{\mathcal{C}}

\renewcommand{\implies}{\Rightarrow}

\DeclareMathOperator{\Isom}{Isom}
\DeclareMathOperator{\proj}{proj}
\DeclareMathOperator{\Aut}{Aut}

\DeclareMathOperator{\length}{length}

\DeclareMathOperator{\Image}{Im}
\DeclareMathOperator{\lk}{Link}

\DeclareMathOperator{\Stab}{Stab}
\newcommand{\Haus}{\text{Haus}}

\DeclareMathOperator{\Cusp}{Cusp}

\DeclareMathOperator{\diam}{diam}

\DeclareMathOperator{\join}{join}

\usepackage{ifthen}

\newcommand{\showcomments}{yes}

\newsavebox{\commentbox}
{\ifthenelse{\equal{\showcomments}{yes}}%
{\footnotemark{}
\begin{lrbox}{\commentbox}
\begin{minipage}[t]{1.25in}\raggedright\sffamily\upshape\tiny
\footnotemark[\arabic{footnote}]}%
{\begin{lrbox}{\commentbox}}}%
{\ifthenelse{\equal{\showcomments}{yes}}%
{\end{lrbox}\end{minipage}\marginpar{\usebox{\commentbox}}}%
{\end{lrbox}}}

\newcounter{acomments}

\newcounter{hcomments}

\usepackage{import}
\usepackage{xifthen}
\usepackage{pdfpages}
\usepackage{transparent}
 
\newcommand{\incfig}[1]{%
	\def\svgwidth{\columnwidth}
	\begingroup%
  \makeatletter%
  \providecommand\color[2][]{%
    \errmessage{(Inkscape) Color is used for the text in Inkscape, but the package 'color.sty' is not loaded}%
    \renewcommand\color[2][]{}%
  }%
  \providecommand\transparent[1]{%
    \errmessage{(Inkscape) Transparency is used (non-zero) for the text in Inkscape, but the package 'transparent.sty' is not loaded}%
    \renewcommand\transparent[1]{}%
  }%
  \providecommand\rotatebox[2]{#2}%
  \newcommand*\fsize{\dimexpr\f@size pt\relax}%
  \newcommand*\lineheight[1]{\fontsize{\fsize}{#1\fsize}\selectfont}%
  \ifx\svgwidth\undefined%
    \setlength{\unitlength}{356.80129164bp}%
    \ifx\svgscale\undefined%
      \relax%
    \else%
      \setlength{\unitlength}{\unitlength * \real{\svgscale}}%
    \fi%
  \else%
    \setlength{\unitlength}{\svgwidth}%
  \fi%
  \global\let\svgwidth\undefined%
  \global\let\svgscale\undefined%
  \makeatother%
  \begin{picture}(1,0.56431541)%
    \lineheight{1}%
    \setlength\tabcolsep{0pt}%
    \put(0,0){\includegraphics[width=\unitlength,page=1]{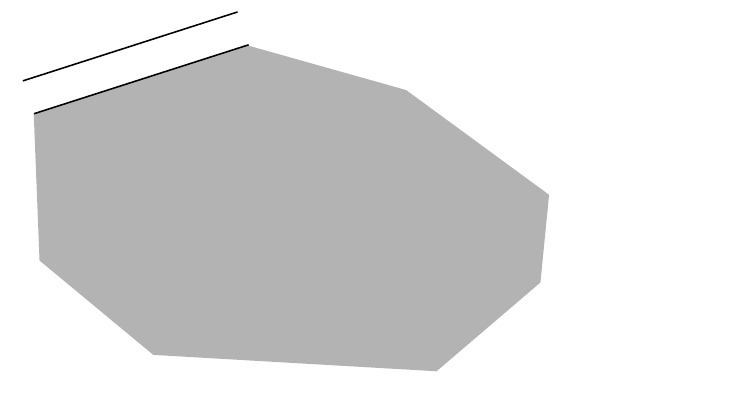}}%
    \put(0.34822336,0.50409119){\color[rgb]{0,0,0}\makebox(0,0)[lt]{\lineheight{1.25}\smash{\begin{tabular}[t]{l}$X_{e}$\end{tabular}}}}%
    \put(0.74675786,0.27297166){\color[rgb]{0,0,0}\makebox(0,0)[lt]{\lineheight{1.25}\smash{\begin{tabular}[t]{l}$X_{e_1}$\end{tabular}}}}%
    \put(0.3310658,0.54409406){\color[rgb]{0,0,0}\makebox(0,0)[lt]{\lineheight{1.25}\smash{\begin{tabular}[t]{l}$X_{\bar e}$\end{tabular}}}}%
    \put(0,0){\includegraphics[width=\unitlength,page=2]{path.pdf}}%
    \put(0.79951428,0.34753741){\color[rgb]{0,0,0}\makebox(0,0)[lt]{\lineheight{1.25}\smash{\begin{tabular}[t]{l}$X_{\bar {e_1}}$\end{tabular}}}}%
    \put(0,0){\includegraphics[width=\unitlength,page=3]{path.pdf}}%
    \put(0.73119287,0.18465701){\color[rgb]{0,0,0}\makebox(0,0)[lt]{\lineheight{1.25}\smash{\begin{tabular}[t]{l}$X_{e_2}$\end{tabular}}}}%
    \put(0.81130669,0.12114732){\color[rgb]{0,0,0}\makebox(0,0)[lt]{\lineheight{1.25}\smash{\begin{tabular}[t]{l}$X_{\bar {e_2}}$\end{tabular}}}}%
    \put(0,0){\includegraphics[width=\unitlength,page=4]{path.pdf}}%
    \put(0.15442737,0.01119734){\color[rgb]{0,0,0}\makebox(0,0)[lt]{\lineheight{1.25}\smash{\begin{tabular}[t]{l}$X_{\bar {e_{n}}}$\end{tabular}}}}%
    \put(0,0){\includegraphics[width=\unitlength,page=5]{path.pdf}}%
    \put(0.21139305,0.09133196){\color[rgb]{0,0,0}\makebox(0,0)[lt]{\lineheight{1.25}\smash{\begin{tabular}[t]{l}$X_{e_{n}}$\end{tabular}}}}%
    \put(0,0){\includegraphics[width=\unitlength,page=6]{path.pdf}}%
    \put(0.37429857,0.24754846){\color[rgb]{0,0,0}\makebox(0,0)[lt]{\lineheight{1.25}\smash{\begin{tabular}[t]{l}$X_v$\end{tabular}}}}%
    \put(0.12071911,0.41187492){\color[rgb]{0,0,0}\makebox(0,0)[lt]{\lineheight{1.25}\smash{\begin{tabular}[t]{l}$y$\end{tabular}}}}%
    \put(0.26506039,0.45619077){\color[rgb]{0,0,0}\makebox(0,0)[lt]{\lineheight{1.25}\smash{\begin{tabular}[t]{l}$x$\end{tabular}}}}%
    \put(0.39444127,0.33769402){\color[rgb]{0,0,0}\makebox(0,0)[lt]{\lineheight{1.25}\smash{\begin{tabular}[t]{l}$\gamma_0$\end{tabular}}}}%
    \put(0.11961021,0.29772757){\color[rgb]{0,0,0}\makebox(0,0)[lt]{\lineheight{1.25}\smash{\begin{tabular}[t]{l}$\gamma_n$\end{tabular}}}}%
    \put(0.26457892,0.13220584){\color[rgb]{0,0,0}\makebox(0,0)[lt]{\lineheight{1.25}\smash{\begin{tabular}[t]{l}$\gamma_{n-1}$\end{tabular}}}}%
    \put(0.54817199,0.11170722){\color[rgb]{0,0,0}\makebox(0,0)[lt]{\lineheight{1.25}\smash{\begin{tabular}[t]{l}$\gamma_2$\end{tabular}}}}%
    \put(0.60994834,0.25599548){\color[rgb]{0,0,0}\makebox(0,0)[lt]{\lineheight{1.25}\smash{\begin{tabular}[t]{l}$\gamma_1$\end{tabular}}}}%
    \put(0.59189138,0.43888635){\color[rgb]{0,0,0}\makebox(0,0)[lt]{\lineheight{1.25}\smash{\begin{tabular}[t]{l}$\alpha_1$\end{tabular}}}}%
    \put(0.73020064,0.13079308){\color[rgb]{0,0,0}\makebox(0,0)[lt]{\lineheight{1.25}\smash{\begin{tabular}[t]{l}$\alpha_2$\end{tabular}}}}%
    \put(0.14591046,0.07026407){\color[rgb]{0,0,0}\makebox(0,0)[lt]{\lineheight{1.25}\smash{\begin{tabular}[t]{l}$\alpha_{n}$\end{tabular}}}}%
    \put(0.72504948,0.45397502){\color[rgb]{0,0,0}\makebox(0,0)[lt]{\lineheight{1.25}\smash{\begin{tabular}[t]{l}$\delta_1$\end{tabular}}}}%
    \put(0.79440319,0.00493716){\color[rgb]{0,0,0}\makebox(0,0)[lt]{\lineheight{1.25}\smash{\begin{tabular}[t]{l}$\delta_2$\end{tabular}}}}%
    \put(0.00949809,0.0307307){\color[rgb]{0,0,0}\makebox(0,0)[lt]{\lineheight{1.25}\smash{\begin{tabular}[t]{l}$\delta_{n}$\end{tabular}}}}%
    \put(0.70216364,0.35620064){\color[rgb]{0,0,0}\makebox(0,0)[lt]{\lineheight{1.25}\smash{\begin{tabular}[t]{l}$\beta_1$\end{tabular}}}}%
    \put(0.64088807,0.04676999){\color[rgb]{0,0,0}\makebox(0,0)[lt]{\lineheight{1.25}\smash{\begin{tabular}[t]{l}$\beta_2$\end{tabular}}}}%
    \put(0.0501156,0.1475382){\color[rgb]{0,0,0}\makebox(0,0)[lt]{\lineheight{1.25}\smash{\begin{tabular}[t]{l}$\beta_{n}$\end{tabular}}}}%
  \end{picture}%
\endgroup%

}

\begin{document}

\title{Quasi-isometric rigidity of extended admissible groups}

\author{Alex Margolis}
\address{Alex Margolis, Department of Mathematics, The Ohio State University,  Mathematics Tower,  231 W 18th Ave,  Columbus,  OH  43210, USA}
\email{margolis.93@osu.edu}

\author{Hoang Thanh Nguyen}
\address{Hoang Thanh Nguyen, Department of Mathematics, FPT University, Hoa Hai ward, Ngu Hanh Son district, Da Nang, Vietnam}
\email{hoangnt63@fe.edu.vn}

\date{January 5, 2024}

\begin{abstract}
	We introduce the class of extended admissible groups, which include both  fundamental groups of non-geometric 3-manifolds and Croke--Kleiner admissible groups. We show that the class of  extended admissible groups is quasi-isometrically rigid.
\end{abstract}

\maketitle
\begin{center}
\end{center}

\section{Introduction}
A central idea in geometric group theory is that a finitely generated group
equipped with the word metric is a geometric object in its own right. This
metric is well-defined up to quasi-isometry. Geometric group theory explores the connection between algebraic and large-scale geometric properties of
finitely generated groups.
One of the fundamental questions of geometric group theory, posed by Gromov~\cite{gromov1993asymptotic}, is the following.

\begin{ques}[Quasi-isometric rigidity]\label{ques:QI}
	Given a class $\mathcal{C}$ of finitely generated groups, is any finitely generated group quasi-isometric to a group in $\mathcal{C}$  virtually isomorphic to a group in $\mathcal{C}$?
\end{ques}

The study of quasi-isometric rigidity is a major focus of geometric group theory. This has led to the emergence of numerous new concepts that have far-reaching implications. These include connections between the topological notion of ends and algebraic splittings~\cite{Sta68, Dun85}, the advancement of quasi-conformal geometry~\cite{Tuk88, Gab92, CJ94, Sch95, BP00}, and the analysis of asymptotic cones~\cite{Gro81, vdDW84, kapovichleeb1997Quasiisometries}.

The primary objective of this paper is to address Question~\ref{ques:QI} for this class of \emph{extended admissible groups}. These are groups possessing a similar graph of groups structure to that of non-geometric 3-manifolds, which will be briefly discussed.

\subsection{Motivation}
We assume 3-manifolds are compact, connected, orientable and irreducible,  with empty or toroidal boundary.  By the geometrization theorem  of Perelman %
and Thurston, a 3-manifold $M$ is either  \emph{geometric}, in the sense that its interior admits one of the following geometries: $S^3$, $\mathbb{E}^3$, $\mathbb{H}^3$, $S^{2} \times \mathbb{R}$, $\mathbb{H}^{2} \times \mathbb{R}$, $\widetilde{SL(2, \mathbb{R})}$, Nil, and Sol; or the manifold $M$ is {\it non-geometric}.
The class of 3-manifold groups is known to be quasi-isometrically rigid. Kapovich--Leeb provide a complete solution to Question~\ref{ques:QI} for fundamental groups of non-geometric 3-manifolds~\cite{kapovichleeb1997Quasiisometries}. For more general 3-manifold groups, including those with surface boundaries, see~\cite{HL20}.

Let $M$ be a non-geometric 3-manifold. The torus decomposition of $M$ yields a nonempty minimal union $\mathcal{T} \subset M$ of disjoint essential tori, unique up to isotopy, such that each component $M_v$ of $M \backslash \mathcal{T}$, called a \emph{piece}, is either  Seifert fibered or hyperbolic.
There is an induced graph of groups decomposition $\mathcal{G}$ of $\pi_1(M)$ with underlying graph $\Gamma$ as follows. For each piece $M_v$, there is a vertex $v$ of $\Gamma$  with vertex group $\pi_1(M_{v})$. For each torus $T_e\in \mathcal T$ contained in the closure of pieces $M_v$ and $M_{w}$, there is an edge $e$ of $\Gamma$ between vertices $v$ and $w$. The associated edge group is $\pi_1(T_e)\cong \Z^2$ and the edge monomorphisms are the maps induced by inclusion.

A \emph{$\Z$-by-hyperbolic group} is a finitely generated group $G$ containing an infinite cyclic normal
subgroup $H\cong \Z$ such that the  quotient $G/H$ is a non-elementary hyperbolic group. Each Seifert fibered piece $M_v$ in the JSJ decomposition of $M$ admits a Seifert fibration over a hyperbolic 2-orbifold $\Sigma_v$; thus $\pi_1(M_v)$  contains an infinite cyclic normal subgroup $\Z$ such that the quotient $\pi_1(M_v)/ \Z$ is $\pi_1(\Sigma_v)$. In particular, $\pi_1(M_v)$ is a $\Z$-by-hyperbolic group.
If $M_v$ is a hyperbolic piece, then $\pi_1(M_v)$ is hyperbolic relative to  $\{\pi_1(T_1), \ldots, \pi_1(T_\ell) \}$, where $\{T_1, \ldots, T_{\ell} \}$ is the collection of boundary tori of $M_v$.

Croke--Kleiner defined the class of admissible groups, which have a graph of groups decomposition generalizing that of graph manifolds~\cite{CK02}. In this paper, we work with the more general class of
\emph{extended admissible groups}, which possess a graph of groups decomposition generalizing that of any non-geometric 3-manifold. In an extended admissible group, we allow any $\Z$-by-hyperbolic group instead of a Seifert fibered piece, and we allow any toral relatively hyperbolic group instead of a hyperbolic piece.   For the precise definition of extended admissible groups, we refer the reader to Definition~\ref{defn:extended}.

There has been a recent focus in geometric group theory of studying groups that are not hyperbolic, but exhibit some features of coarse negative curvature. The class of non-geometric 3-manifold groups, and more generally of (extended) admissible groups, provide a rich source of such groups, being among the prototypical examples of acylindrically hyperbolic groups and of hierarchically hyperbolic spaces and groups~\cite{minasyanosin2015acylindrical, behrstockhagensistosro2019hierarchically, HRSS22}. Determining to what extent these forms of coarse negative curvature are invariant under quasi-isometry, as well as studying quasi-isometric rigidity and classification problems for such groups,  is an important problem in the area.

\subsection{Quasi-isometric rigidity}
The main result of this paper is the following quasi-isometric rigidity theorem for extended admissible groups:

\begin{thm}\label{thm:CKrigidity_qi}
	Let $G$ be an extended admissible group. If  $G'$ is a finitely generated group quasi-isometric to $G$, then $G'$ has a finite index subgroup that is an extended admissible group.
\end{thm}

The main ingredient needed to prove Theorem~\ref{thm:CKrigidity_qi} is the following result of independent interest, demonstrating that quasi-isometries preserve the graph of groups decomposition of an extended admissible group.
\begin{thm}\label{thm:CKrigidity_graphofgroups}
	Let $G$ be an extended admissible group and let $(X,T)$ be the associated tree of spaces.
	For every quasi-isometry $f\colon X \to X$, there is a tree isomorphism $f_*:T\to T$ such that for every  vertex or edge space $X_x$ of $X$,  $f(X_x)$ is at finite Hausdorff distance from $X_{f_*(x)}$.
\end{thm}
Theorem~\ref{thm:CKrigidity_graphofgroups} is deduced from Theorem~\ref{thm:vertextovertexextend}, a more  quantitative analog version of Theorem~\ref{thm:CKrigidity_graphofgroups}. Theorems~\ref{thm:CKrigidity_qi} and~\ref{thm:CKrigidity_graphofgroups} generalize the main result of Kapovich--Leeb in~\cite{kapovichleeb1997Quasiisometries}, and answers~\cite[Question~1.8]{NY23}.

Our proof of Theorem~\ref{thm:CKrigidity_graphofgroups} makes use of asymptotic cones, which were also used by   Kapovich--Leeb~\cite{kapovichleeb1997Quasiisometries}. However, working in the setting of extended admissible groups presents some genuine difficulties. Indeed, a key point of Kapovich--Leeb’s proof is that up to quasi-isometry, it can be assumed  non-geometric 3-manifolds are CAT(0) and that   ultralimits of geometric pieces are convex~\cite{KL98}.  In our setting, we can no longer appeal to such a result,  since it is not known whether an arbitrary  $\Z$-by-hyperbolic group is quasi-isometric to a CAT(0) space.   The details of our proof are thus different from that of Kapovich--Leeb, and a large part of our proof of Theorem~\ref{thm:CKrigidity_graphofgroups} is devoted to overcoming the fact that vertex and edge spaces of the tree of spaces $(X,T)$ are not typically convex or quasi-convex.

Theorems~\ref{thm:CKrigidity_qi} and~\ref{thm:CKrigidity_graphofgroups} fit into a long history of showing that splitting as a graph of groups in a prescribed manner is a quasi-isometry invariant, including work of Stallings~\cite{Sta68}, Mosher--Sageev--Whyte~\cite{mosher2003quasi, moshersageevwhyte2011quasiactions}, Papasoglu~\cite{papasoglu2005quasiisometry,papasoglu2007group} and Margolis~\cite{margolis2018quasi,margolisxer2021geometry}. These results do not overlap with Theorem~\ref{thm:CKrigidity_graphofgroups} except in a few exceptional cases, e.g.\ where vertex groups are coarse $PD_n$ groups and the results of~\cite{moshersageevwhyte2011quasiactions} can be applied.

\subsection{Applications}
We close the paper with some applications of our main results.
\subsubsection{Quasi-isometric classification}
A well-known  companion problem to Question~\ref{ques:QI} is the following:
\begin{ques}[Quasi-isometric classification]\label{ques:QI_class}
	Given a class $\mathcal{C}$ of finitely generated groups, determine when two elements of $\mathcal{C}$ are quasi-isometric.
\end{ques}
Behrstock--Neumann classified fundamental groups of non-geometric 3-manifolds up to quasi-isometry~\cite{BN08,behrstockneumann2012quasiisometric}, and made crucial use of the special case of Theorem~\ref{thm:CKrigidity_graphofgroups} proven by Kapovich--Leeb~\cite{kapovichleeb1997Quasiisometries}. It is thus natural to study the quasi-isometry classification of extended admissible groups using Theorem~\ref{thm:CKrigidity_graphofgroups} as a starting point.

Each vertex group of an extended admissible group is either  $\Z$-by-hyperbolic or is relatively hyperbolic; we call these  \emph{type $\mathcal{S}$} and \emph{type $\mathcal{H}$} respectively.  The \emph{hyperbolic quotient} of a type $\cS$ vertex group $G_v$ is the quotient of $G_v$  by an infinite cyclic normal subgroup; this hyperbolic quotient is well-defined up to a finite normal subgroup.  It follows from the work of Kapovich--Kleiner--Leeb that two type $\cS$ vertex groups are quasi-isometric if and only if their hyperbolic quotients are quasi-isometric~\cite{kapovichkleinerleeb1998quasiisometries}; see also~\cite[Theorem A]{Mar22}. Combining this with Theorem~\ref{thm:CKrigidity_graphofgroups}, we deduce the following necessary criterion for extended admissible groups to be quasi-isometric:
\begin{cor}\label{cor:qi_class}
	Let $G$ and $G'$ be extended admissible groups. If $G$ and $G'$ are quasi-isometric, then:
	\begin{enumerate}
		\item Every type $\cH$ vertex group of $G$ is quasi-isometric to a type $\cH$ vertex group of $G'$ and vice-versa.
		\item Every hyperbolic quotient of a type $\cS$ vertex group of $G$ is quasi-isometric to a hyperbolic quotient of a  type $\cS$ vertex group of $G'$ and vice-versa.
	\end{enumerate}
\end{cor}

While this is far from a complete quasi-isometric classification, it does demonstrate that there are infinitely many quasi-isometry classes of extended admissible groups. In particular, a finitely generated group quasi-isometric to an extended admissible group containing at least one type $\cH$ vertex (resp.\ at least one type $\cS$ vertex) must also be an extended admissible group containing at least one type $\cH$ vertex (resp.\ at least one type $\cS$ vertex).

\subsubsection{Admissible groups with hyperbolic manifold quotient groups}
Theorem~\ref{thm:CKrigidity_graphofgroups}  implies that a quasi-isometry between extended admissible groups $G$ and $G'$ induces quasi-isometries between vertex groups of $G$ and of $G'$ that coarsely preserve the collection of incident edge spaces. In certain situations, a quasi-isometry that coarsely preserves some distinguished collection of subspaces is much more rigid than an arbitrary quasi-isometry.  This phenomenon,  called \emph{pattern rigidity}, is originally due to Schwartz~\cite{schwartz1997symmetric} in the context of geodesics in hyperbolic space. By combining Theorem~\ref{thm:CKrigidity_graphofgroups} with Schwartz's pattern rigidity theorem~\cite{schwartz1997symmetric} we deduce the following:

\begin{cor}\label{cor:comm}
	Let $G$ be an extended admissible group such that all vertex groups $G_v$ are of type $\cS$ and have hyperbolic quotient  $Q_v$ isomorphic to the fundamental group of a closed hyperbolic $n_v$-manifold for some $n_v \ge 3$.

	If $G'$ is a finitely generated group quasi-isometric to $G$,  then $G'$ is an extended admissible group such that all vertex groups of $G'$ are of type $\cS$ with hyperbolic quotient virtually isomorphic to some hyperbolic quotient  $Q_v$ of a vertex group of $G$.
\end{cor}

\subsubsection{Uniform lattice envelopes of extended admissible groups}
We recall that a \emph{uniform lattice} $G$ in a locally compact group $\hat G$ is a discrete cocompact subgroup of $\hat G$. For example, if $G$ acts faithfully, properly, cocompactly and isometrically on a proper metric space $X$, then $G$ is a uniform lattice in $\Isom(X)$, where $\Isom(X)$ is equipped with the compact-open topology. If $G$ is (isomorphic to) a uniform lattice of $\hat G$, then $\hat G$ is called a \emph{uniform lattice envelope} of $G$.  A well-known problem, originating in work of Mostow, is to classify uniform lattice envelopes of a fixed countable group $G$. Progress on this problem has been made by Furman~\cite{furman2001mostowmargulis}, Dymarz~\cite{dymarz2015envelopes}, Bader--Furman--Sauer~\cite{baderfurmansauer2020lattice} and Margolis--Shepherd--Stark--Woodhouse~\cite{MSSW23}.

We fix an extended admissible group $G$ and let $T$ be the associated Bass--Serre tree of $G$.
Combining Theorem~\ref{thm:CKrigidity_graphofgroups} with the argument in the proof of~\cite[Corollary~11.12]{MSSW23} (see also~\cite{furman2001mostowmargulis}) we deduce the  following:

\begin{cor}\label{cor:ctsaction}
	Let $G$ be an extended admissible group with the associated Bass--Serre tree $T$. If $\hat G$ is a uniform lattice envelope of $G$, then the action of $G$ on $T$ extends to a continuous action of $\hat G$ on $T$. In particular, for each vertex or edge $x$ of $T$, $\Stab_{\hat G}(x) $ is a locally compact group  containing $\Stab_{G}(x)$ as a uniform lattice.
\end{cor}
In~\cite{MSSW23}, statements similar to Corollary~\ref{cor:ctsaction} are an essential ingredient in proving action rigidity for many classes of groups.
Although not pursued in this article, we believe Corollary~\ref{cor:ctsaction} has significant potential in proving similar action rigidity theorems  for certain extended admissible groups.

\subsection*{Acknowledgments} We thank Chris Hruska for
useful conversations.

\section{Preliminaries}
In this section, we review some concepts in geometric group theory that will be used throughout the paper.
\subsection{Coarse geometry}

Let $X$ and $Y$ be metric spaces and $f$ be a map from $X$ to $Y$.
\begin{enumerate}
	\item We say that $f$ is \emph{$(K,A)$-coarse Lipschitz}  if for all $x, y\in X$, \[d(f(x), f(y)) \le K d(x, y) + A.\]
	\item We say that $f$ is a \emph{$(K,A)$--quasi-isometric embedding} if  for all $x, y\in X$,
	      \[
		      \frac{1}{K} d(x, x') - A \le d(f(x), f(x')) \le K d(x,x') + A .
	      \]

	\item We say that $f$ is a \emph{$(K,A)$--quasi-isometry} if it is a $(K,A)$--quasi-isometric embedding such that $Y = N_{A}(f(X))$.

	\item We say that $f$ is a \emph{$K$-bi-Lipschitz equivalence}, if it is a $(K, 0)$-quasi-isometry.

	\item We say two quasi-isometries $f, g \colon X \to Y$ are \emph{$A$--close} if
	      \[
		      \sup_{x \in X} d_{Y} (f(x), g(x)) \le A
	      \] and are \emph{close} if they are $A$--close for some $A \ge 0$.
\end{enumerate}
We say $f$ is \emph{coarse Lipschitz} if it is  $(K,A)$-coarse Lipschitz for some $K\geq 1$ and $A\geq 0$. We define what it means for $f$ to be a quasi-isometric embedding, quasi-isometry etc.\ similarly.

\begin{defn}[Quasi-action]
	If $G$ is a group and $X$ is a metric space, then a \emph{$(K,A)$--quasi-action} of $G$ on $X$ is a collection of maps $\{f_g\}_{g \in G}$ such that
	\begin{itemize}
		\item For every $g$ in $G$, the map $f_g \colon X \to X$ is a $(K,A)$--quasi-isometry.

		\item For every $g, h \in G$,  $f_{gh}$ is $A$--close to $f_{g} \circ f_{h}$.

		\item $f_{1_G}$ is $A$--close to the identity on $X$.
	\end{itemize}
	A \emph{quasi-action} of a group $G$ on $X$ is a $(K,A)$--quasi-action of $G$ on $X$ for some $K \ge 1$ and $A \ge 0$.\end{defn}

\subsection{Bass--Serre theory}\label{sec:TofS}

We assume familiarity with Bass--Serre theory; see~\cite{SW79} for details. However, to fix notation and terminology, we give some brief definitions.

We first establish some terminology regarding graphs. A \emph{graph} $\Gamma$ consists of a set $V\Gamma$ of vertices, a set $E\Gamma$ of oriented edges, and maps $\iota,\tau:E\Gamma\to V\Gamma$. There is a fixed-point free involution $E\Gamma\to E\Gamma$, taking an edge $e\in E\Gamma$ such that $\iota e=v$ and $\tau e=w$ to an edge $\bar e$ satisfying $\iota \bar e=w$ and $\tau \bar e=v$. We also write $e_+$ and $e_-$ to denote $\tau e$ and $\iota e$ respectively. An \emph{unoriented edge} of $\Gamma$ is the pair $\{e,\bar e\}$. If $v$ is a vertex, we define $\lk(v)=\{e\in E\Gamma\mid e_-=v\}$.

Each connected graph can be identified with a  metric space by equipping its topological realization with the path metric in which each edge has length one. A \emph{combinatorial path} in $X$ is a path $p:[0,n]\to X$ for some $n\in \N$ such that for every integer $i$, $p(i)$ is a vertex, and $p|_{[i,i+1]}$ is either constant or traverses an edge of $X$ at unit speed. Every geodesic between vertices of $X$ is necessarily a combinatorial path.

\begin{defn}
	A \emph{graph of groups} $\mathcal{G} = (\Gamma, \{G_{v}\}, \{G_{e} \}, \{\tau_{e} \})$ consists of the following data:
	\begin{enumerate}
		\item a graph $\Gamma$, called the \emph{underlying graph},
		\item a group $G_v$ for each vertex $v \in V \Gamma$, called a \emph{vertex group},
		\item a subgroup $G_e\leq G_{e_-}$ for each edge $e \in E \Gamma$, called an \emph{edge group},
		\item an isomorphism $\tau_{e} \colon {G}_{e} \to {G}_{\bar e}$  for each $e\in E\Gamma$ such that $\tau^{-1}_e=\tau_{\bar e}$, called an \emph{edge map}.
	\end{enumerate}
\end{defn}
The \emph{fundamental group} $\pi_{1}(\cG)$ of a graph of groups $\cG$ is as defined in~\cite{SW79}.

We use the following notation for trees of spaces as in~\cite{CM17}.

\begin{defn}\label{defn:treeofspaces}
	\emph{A tree of spaces} $X:=X\left(T,\left\{X_v\right\}_{v \in VT},\left\{X_e\right\}_{e \in E T},\left\{\alpha_e\right\}_{e \in ET}\right)$ consists of:
	\begin{enumerate}
		\item a simplicial tree $T$, called the \emph{base tree};
		\item a metric space $X_v$ for each vertex $v$ of $T$, called a \emph{vertex space};
		\item a subspace $X_e \subseteq X_{e_{-}}$ for each oriented edge $e$ (with the initial vertex denoted by $e_{-}$) of $T$, called an \emph{edge space};
		\item maps $\alpha_e: X_e \rightarrow X_{\bar{e}}$ for each edge $e \in E T$, such that $\alpha_{\bar{e}} \circ \alpha_e=\operatorname{id}_{X_e}$ and $\alpha_e \circ \alpha_{\bar{e}}=\mathrm{id}_{X_{\bar{e}}}$.
	\end{enumerate}
	We consider $X$ as a metric space as follows: we take the disjoint union of all the $X_v$ and then, for all unoriented edges $\{e, \bar{e}\}$ and every $x \in X_e$, we attach a unit interval between $x \in X_e$ and $\alpha_e(x) \in X_{\bar{e}}$. Each edge and vertex space can be naturally identified with a subspace of $X$.
\end{defn}
We typically omit the data $X_v$, $X_e$ and $\alpha_e$ from the notation and write a tree of spaces as  the pair $(X,T)$, or simply as a space $X$.
We consider $X$ as a metric space by equipping it  with the induced path metric.
We now explain how to associate a tree of spaces to a graph of finitely generated groups. Although this construction is standard, the details and notation are not completely standardized, so we describe it in detail.

Let $\mathcal{G} = (\Gamma, \{G_{v}\}, \{G_{e} \}, \{\tau_{e} \})$ be a graph of finitely generated groups. We recall the associated Bass--Serre tree $T$ is constructed so that vertices (resp.\ edges) of $T$ correspond to left cosets of vertex (resp.\ edge) groups of $\cG$.

We now describe a tree of spaces $X$.  For each $x\in V\Gamma \sqcup E\Gamma$, we fix a finite generating set $S_x$ of $G_x$, chosen such that $\tau_e(S_e)=S_{\bar e}$.
We now define a graph $W$ with vertex set $V\Gamma\times G$ and edge set \[\{((v, g),(v,gs))\mid g\in G, s\in S_v\}.\] The components of $W$ are in bijective correspondence with left cosets of vertex groups of $\cG$, and hence with vertices of $T$. If $\tilde{v}\in VT$ corresponds to the coset $gG_v$, we define $X_{\tilde{v}}$  to be the  component of $W$ with vertex set $\{(v,h)\mid h\in gG_v\}$. We note that the component of $W$ corresponding to a coset $gG_v$ is isometric to the Cayley graph of $G_v$ with respect to $S_v$.

Suppose $\tilde{e}\in ET$ corresponds to a coset $gG_e$. By the definition of $T$, if $v=e_-$ and $w=e_+$, then $\tilde{v}\coloneqq\tilde{e}_-$ and $\tilde{w}\coloneqq\tilde{e}_+$ correspond to the cosets $gG_v$ and $gG_w$. We define the edge space $X_{\tilde{e}}$ to  be \[\{(v,h)\mid h\in gG_e\}\subseteq X_{\tilde{v}}.\]  The attaching map $\alpha_{\tilde e}:X_{\tilde{e}}\to X_{\tilde{w}}$ is defined by  $\alpha_{\tilde{e}}:(v,h)\mapsto (w,g\tau_e({g^{-1}h}))$, where $\tau_e:G_e\to G_{\bar{e}}\leq G_w$ is the edge map of $\cG$. Finally, we equip each $X_{\tilde{e}}$ with the word metric with respect to $S_e$. (More precisely, we require that the map $X_{\tilde{e}}\xrightarrow{(v,h)\mapsto g^{-1}h} G_e$ is an isometry when $G_e$ is equipped with the word metric with respect to $S_e$.)

\begin{defn}\label{defn:treeofspacesassociated}
	Given a graph of finitely generated groups $\cG$, the tree of spaces $X$ constructed above  is \emph{the tree of spaces associated with the graph of groups $\cG$}.
\end{defn}

The tree of spaces $X$ is a proper geodesic metric space (see Lemma 2.13 of~\cite{CM17}). The natural action of $G$ on $W$ (fixing the $V\Gamma$ factor) induces an action of $G$ on $X$. Applying the  Milnor--Schwarz lemma we deduce:
\begin{prop}[Section 2.5 of~\cite{CM17}]\label{prop:treeofspacesBST}
	Suppose $G$, $T$ and $X$ are as above. Then there exists a quasi-isometry $f: G \rightarrow X$ and $A\geq 0$ such that
	$d_{\mathrm{Haus }}\left(f\left(gG_{x}\right), X_{\tilde{x}}\right) \leq A$  for all $\tilde{x} \in V T\sqcup  E T$, where $\tilde{x}$ corresponds to the coset $gG_x$.
\end{prop}

The following lemma is presumably well-known, but we couldn't find an explicit proof in the literature. We provide the proof here for the benefit of the reader.
\begin{lem}\label{lem:vertex_edge_QIembeddings}
	Let $\cG$ be a finite graph of finitely generated groups, and let $G=\pi_1(\cG)$. If every edge group of $\cG$ is quasi-isometrically embedded in $G$, then so is every vertex group.
\end{lem}
\begin{proof}
	Let $\Gamma$ be the underlying graph of $\cG$. For each edge group $G_e$, pick a finite generating set $S_e$ and let $d_e$ be the associated word metric on $G_e$.
	For each vertex group $G_v$ of $\cG$, pick a finite generating set $S_v$ containing $\tau_e(S_e)$ for each edge $e$ with $e_+=v$. Let $d_v$  denote the corresponding word metric on $G_v$. Then $G$ has a finite generating set of the form $S=\bigcup_{v\in V\Gamma}S_v \cup S_0$, where $S_0$ consists of stable letters corresponding to edges outside a spanning tree of $\Gamma$. Let $d$ denote the corresponding word metric on $G$.

	Since each edge group is quasi-isometrically embedded, there is a constant  $K$ such that for each $e\in E(\Gamma)$ and $g\in G_e$, we have $d_e(1,g)\leq K d(1,g)$.
	Pick a vertex $v\in V\Gamma$ and $g\in G_v$. Let $w$ be a word  in $S$ of length $d(1,g)$ with $w=_Gg$. We can write $w=w_0r_1w_1\dots r_nw_n$, where $n\leq d(1,g)$, each  $w_i$ is a word in $S_v$ and each $r_i$ is a word in $\bigcup_{v'\neq v}S_{v'} \cup S_0$. Using normal forms for graphs of groups, we deduce each $r_i\in \tau_e(G_e)$  for some $e\in E(\Gamma)$ with $e_+=v$. Since $S_v$ contains $\tau_e(S_e)$, there is a  word $\hat r_i$ in $S_v$ of length $d_e(1,r_i)\leq Kd(1,r_i)$ with $\hat r_i\stackrel{G}{=} r_i$. Thus $w_0\hat r_1w_1\dots \hat r_nw_n\stackrel{G}{=}g$ is a word in $S_v$ of length at most $Kd(1,g)$. Therefore, $d_v(1,g)\leq Kd(1,g)$ as required.
\end{proof}

The notion of betweenness in a tree $T$ is defined as follows.
\begin{defn}
	If $e\in ET$, let $\mu_e$ be the point of (the metric realization of) $T$ which lies at distance 1/3 along the edge from $e_-$ to $e_+$, and let $\mu_v=v$ for all $v\in VT$. Given $a,b,c\in VT\sqcup ET$, we say \emph{$b$ is strictly between $a$ and $c$} if $\mu_a$ and $\mu_c$ lie in different components of $T\backslash \mu_b$.  We say \emph{$b$ is between $a$ and $c$} if $b$ is either strictly between $a$ and $b$ or is equal to one of  $a$ or $b$.
\end{defn}
In particular, if $b$ is strictly between $a$ and $c$, then $a$, $b$ and $c$ are distinct. We also note that if $e\in ET$, then $e$ is strictly between $e_-$ and $\bar e$. If $(X,T)$ is a tree of spaces and $b$ is between $a$ and $c$, then any path from $X_a$ to $X_c$ must intersect $X_b$.

\subsection{Asymptotic cones}
This section reviews the background on asymptotic cones, a tool used to prove Theorem~\ref{thm:CKrigidity_graphofgroups}. The material presented here is well-known and can be found in~\cite{KL98,KL96,dructusapir2005treegraded}.

\begin{defn}
	A \emph{non-principal ultrafilter} $\omega$ over $\N$ is a  collection of subsets of $\N$ such that the following conditions hold.
	\begin{enumerate}
		\item If $A,B\in \omega$, then $A\cap B\in\omega$.
		\item If $A\in\omega$, and $A \subseteq B \subseteq \N$ then $B\in\omega$.
		\item For every $A \subseteq \N$,  either $A \in \omega$ or $\N - A \in \omega$.
		\item No finite subset of $\N$ is in $\omega$.
	\end{enumerate}
\end{defn}

Fix a non-principal ultrafilter $\omega$ over $\N$. We say a statement $P_i$ depending on $i \in \N$ holds \emph{$\omega$-almost surely} if the set of indices such that $P_i$ holds belong to $\omega$.
If
$(x_i)$ is a sequence of points in a topological space $X$, we write $\lim_{\omega} x_i = x_\infty$ if for every neighborhood $U$ of $x_\infty$,  $x_i \in U$  $\omega$-almost surely.

Fix a sequence $(X_i,b_i,d_i)$ of based metric spaces, i.e.\ $(X_i,d_i)$ is a metric space and $b_i\in X_i$. A sequence $(x_i)$, where each $x_i\in X_i$, is \emph{$\omega$-admissible} if $\lim_\omega d_i(x_i,b_i)<\infty$. We define an equivalence relation $\sim$ on $\omega$-admissible sequences by $(x_i)\sim (y_i)$  if $\lim_\omega d_i(x_i,y_i)=0$. The \emph{ultralimit} $\lim_\omega(X_i,b_i,d_i)$ is defined to be the set of equivalences classes of $\omega$-admissible sequences equipped with the metric $d_\omega([(x_i)],[(y_i)])=\lim_\omega d_i(x_i,y_i)$. The ultralimit of a sequence of complete metric spaces is complete.

\begin{defn}
	If $(X,d)$ is a  metric space, $(b_i)$ is a sequence of basepoints in $X$, and  $(\lambda_i)$ is a sequence in $\R_{>0}$ such that $\lim_i\lambda_i=\infty$, we define the \emph{asymptotic cone}
	\[
		X_\omega((b_i),(\lambda_i)):=\lim_\omega\left(X,b_i,\frac{d}{\lambda_i}\right)
	\]
	When unambiguous, we denote $X_\omega((b_i),(\lambda_i))$ by $X_\omega$.
\end{defn}
If $X$ is cocompact, i.e.\ $\Isom(X)$ acts cocompactly on $X$, then $X_\omega$ is homogeneous and the isometry type of $X_\omega$ doesn't depend on the choice of basepoints.

The following lemma is well-known.

\begin{lem}\label{lem:induced_bilip}
	Let $X_\omega((b_i),(\lambda_i))$ and $Y_\omega((c_i),(\lambda_i))$ be asymptotic cones of $X$ and $Y$.  If $(f_i \colon X_i \to Y_i)$ is a sequence of $(K,A)$-coarse Lipschitz maps such that  $\lim_\omega(\frac{1}{\lambda_i}d_Y(f_i(b_i),c_i))<\infty$, then $(f_i)$ induces a $K$-Lipschitz map $f_\omega:X_\omega\to Y_\omega$ given by $f_\omega([x_i])=([f_i(x_i)])$.

	Moreover, if each $f_i$ is a $(K,A)$-quasi-isometric embedding (resp. $(K,A)$-quasi-isometry), then $f_\omega$ is a $K$-bi-Lipschitz embedding (resp. $K$-bi-Lipschitz equivalence).
\end{lem}
\begin{rem}
	If $Y$ is cocompact, then the assumption that $\lim_\omega(\frac{1}{\lambda_i}d_Y(f_i(b_i),c_i))<\infty$ in Lemma~\ref{lem:induced_bilip} is not restrictive, since one can just define a sequence of basepoints in $Y$ by $(f(b_i))$.
\end{rem}

If $G$ is a finitely generated group, it can be equipped with the word metric with respect to a finite generating set. An asymptotic cone $G_\omega$ of $G$ is an asymptotic cone of $G$ equipped with this metric. Since the word metric is well-defined up to bi-Lipschitz equivalence, each asymptotic cone  $G_\omega((b_i),(\lambda_i))$ is well-defined up to bi-Lipschitz equivalence.

\begin{defn}\label{defn:ultralimit}
	Let $X$ be a metric space and let $X_\omega=X_\omega((b_i),(\lambda_i))$ be an asymptotic cone of $X$.
	\begin{enumerate}
		\item If $(A_i)$ is a  sequence of non-empty subsets of $X$, we define
		      \[
			      \lim_{\omega} A_i=\{[(a_i)]\in X_\omega\mid a_i\in A_i \text{ for all $i$}\}.
		      \]
		\item Suppose $\cA$  is a collection of non-empty subsets of $X$. We define
		      \[
			      \cA_\omega\coloneqq \{\lim_\omega A_i\mid \lim_\omega A_i\neq \emptyset \text{ and $A_i\in \cA$ for all $i$} \}.
		      \]
	\end{enumerate}
\end{defn}
\begin{lem}\label{lem:as-cone_subsets}
	Let $X_\omega((b_i),(\lambda_i))$ be an asymptotic cone of $X$ and let $\cA$  be a collection of subsets of $X$.
	Assume that there exist constants $K\geq 1$, $C\geq 0$ and finitely many metric spaces $\mathcal B$ such that for each $A\in \cA$ there is some $B\in \mathcal B$ such that  $A$ is the image of a $(K,C)$-quasi-isometric embedding $f:B\to X$.
	Then every $A_\omega\in \cA_\omega$  is bi-Lipschitz-equivalent to an  asymptotic cone of some $B\in \mathcal B$.
\end{lem}
\begin{rem}
	The hypotheses of Lemma~\ref{lem:as-cone_subsets} are all satisfied when $G$ is a finitely generated group equipped with  the word metric and  $\cA$  consists of all left cosets of finitely many quasi-isometrically embedded subgroups of $G$.
\end{rem}

We will make use of the following properties concerning asymptotic cones.
\begin{prop}
	\phantomsection\label{prop:wellknownfacts}
	\begin{enumerate}
		\item
		      Let $G$ be a finitely generated group and let $G_\omega$ be an asymptotic cone of $G$.
		      If $G\cong \Z^n$, then $G_\omega$ is bi-Lipschitz equivalent to $\E^n$.
		      If $G$ is a non-elementary hyperbolic group, then $G_\omega$ is  a geodesically complete $\R$-tree that branches everywhere.
		\item 	If $X$ and $Y$ are metric spaces, then every  asymptotic cone of $X\times Y$ is isometric to $X_\omega\times Y_\omega$, where  $X_\omega$ and $Y_\omega$ are asymptotic cones of $X$ and $Y$.

		\item
		      Let $n\in \N$. For any $(K,A)$-quasi-isometric embedding $f:\E^n\to \E^n$, there is a $B=B(K,A,n)$ such that $N_B(\Image(f))=\E^n$. In particular, $f$ is a quasi-isometry.
	\end{enumerate}
\end{prop}

\begin{lem}\label{lem:qi_halfspace}
	Let $\E^2_{\geq 0}$ be the half-space $\{(x,y)\in \E^2\mid y\geq 0\}$. There is no quasi-isometric embedding $f:\E^2_{\geq 0}\to \E$.
\end{lem}
\begin{proof}
	If such a quasi-isometric embedding were to exist, then after taking ultralimits, it would induce a  bi-Lipschitz embedding $f_\omega:\E^2_{\geq 0}\to \E$ between asymptotic cones. This cannot be the case, since the invariance of domain theorem ensures that there is no continuous injection from an open subset of $\E^2$ to  $\E$.
\end{proof}

\begin{defn}

	A geodesic metric space $X$ is \emph{tree-graded} with respect to
	a collection of closed geodesic subsets $\{P_i \}_{i\in I}$, called \emph{pieces}, if the following hold:
	\begin{enumerate}
		\item $|P_i\cap P_j|\leq 1$ if $i\neq j$
		\item Any simple geodesic triangle in $X$ is contained in some $P_i$.
	\end{enumerate}

\end{defn}

\begin{defn}
	Let $X$ be a metric space and $\cA$ a collection of subsets of $X$. We say $X$ is \emph{asymptotically tree-graded} with respect to $\cA$ if every asymptotic cone $X_\omega$ of $X$ is tree-graded with respect to $\cA_\omega$.
	Suppose $G$ is a finitely generated group and $\cH$ is a collection of subgroups of $G$.
	A finitely generated group $G$ is said to be \emph{asymptotically tree-graded with respect
		to a collection of subgroups} $\mathcal H$ if $G$ is asymptotically tree-graded with respect to  the collection of all left cosets of subgroups in $\cH$.
\end{defn}

Dru\c{t}u--Sapir obtained  the following characterization of relatively hyperbolic groups, which can be taken as a definition for the purpose of this article.
\begin{thm}[\cite{dructusapir2005treegraded}]\label{thm:drutusapir_treegraded}
	A finitely generated group $G$ is hyperbolic relative to a collection of subgroups $\cH$ if and only if $G$ is asymptotically tree-graded with respect to $\cH$.
\end{thm}
If $G$ is hyperbolic relative to $\cH$,  then elements of $\cH$ are called \emph{peripheral subgroups} of $G$.

We will make use of the following lemma concerning asymptotic cones of relatively hyperbolic groups:
\begin{lem}\label{lem:ultralimitsvscosets}
	Let $G$ be a  finitely generated group that  is hyperbolic relative to a collection $\cH$ of infinite subgroups. Suppose $G_\omega$ is an asymptotic cone of $G$, and $\lim_{\omega}(g_iH_i)={\lim_{\omega}(g'_i{H'}_i)}\neq \emptyset$, where $g_i,g'_i\in G$ and $H_i,H_i'\in \cH$ for all $i$. Then ${g_iH_i}={{g'}_i{H'}_i}$ $\omega$-almost surely.
\end{lem}
\begin{proof}
	Suppose $G_\omega=G_\omega((b_i),(\lambda_i))$.  Let $\cP$ be the set of all left cosets of elements of $\cH$, and for each $P\in \cP$, let $\proj_P:G\to P$ be a closest point projection map. Set $P_i\coloneqq g_iH_i$ and $Q_i\coloneqq g_i'H'_i$. By~\cite[Theorem 2.14]{Sis13}, we can choose a constant $C$ such that the following hold:
	\begin{enumerate}
		\item $\diam(\proj_P(P'))\leq C$ for all distinct $P,P'\in \cP$;
		\item for all $x\in X$, $P\in \cP$ and $p\in P$, $d(x,p)\geq d(x,\proj_P(x))+d(\proj_P(x),p)-C$.
	\end{enumerate}
	Since $\lim_{\omega}P_i=\lim_{\omega}Q_i\neq \emptyset$, we can choose sequences $(x_i)$ and $(y_i)$ such that $x_i\in P_i$, $y_i\in Q_i$ and $[(x_i)]=[(y_i)]\in \lim_{\omega}P_i$. Therefore, $d(x_i,y_i)\leq \lambda_i$ $\omega$-almost surely.

	Assume for contradiction that  $P_i\neq Q_i$ $\omega$-almost surely.
	The choice of $C$ ensures that $d(x_i,\proj_{P_i}(y_i))\leq C+\lambda_i$ $\omega$-almost surely. Since $P_i$ is unbounded and $\diam(\proj_{P_i}(Q_i))\leq C$ $\omega$-almost surely, we can choose $z_i\in P_i\setminus N_{C+\lambda_i}(\proj_{P_i}(Q_i))$ with $d(z_i,x_i)\leq 3C+2\lambda_i+1$ $\omega$-almost surely. Then $\lim_\omega \frac{d(z_i,b_i)}{\lambda_i}\leq \lim_\omega \frac{d(z_i,x_i)}{\lambda_i}+\lim_\omega \frac{d(x_i,b_i)}{\lambda_i}<\infty$ and $\lim_\omega \frac{d(z_i,Q_i)}{\lambda_i}\geq 1$. Hence $[(z_i)]\in \lim_{\omega}P_i\setminus \lim_{\omega}Q_i$, contradicting our assumption $\lim_{\omega}P_i=\lim_{\omega}Q_i$.
\end{proof}

\subsection{Extended admissible groups}

We now define the class of extended admissible groups.

\begin{defn}\label{defn:extended}
	A group $G$ is an  \emph{extended admissible group} if it is the fundamental group of a graph of groups $\cG$ such that:
	\begin{enumerate}
		\item The underlying graph $\Gamma$ of $\cG$ is a connected finite graph with at least one edge, and every edge group is virtually  $\Z^2$.
		\item Each vertex group $G_v$ is one of the following two types:
		      \begin{enumerate}
			      \item  Type $\mathcal{S}$: $G_v$ contains  an infinite cyclic normal subgroup $Z_v \lhd G_v$, such that the quotient $Q_v \coloneqq G_v / Z_v$ is a non-elementary hyperbolic group. We call $Z_v$ and $Q_v$  the \emph{kernel} and \emph{hyperbolic quotient} of $G_v$ respectively.
			      \item Type $\mathcal{H}$: $G_v$ is hyperbolic relative to a collection $\PP_v$ of virtually $\Z^2$-subgroups, where all edge groups incident to $G_v$ are contained in $\PP_v$, and  $G_v$ doesn't split relative to $\PP_v$  over a subgroup of an element of $\PP_v$.
		      \end{enumerate}
		\item  For each vertex group $G_v$, if $e,e'\in \lk(v)$ and $g\in G_v$, then  $gG_eg^{-1}$ is commensurable to $G_{e'}$ if and only if both $e=e'$ and $g\in G_e$.

		\item For every edge group ${ G}_e$ such that $G_{e_{-}}$ and $G_{e_{+}}$ are vertex groups of type $\mathcal{S}$, the subgroup generated by $\tau_{\bar e}(Z_{e_+}\cap {G}_{\bar e})$  and $Z_{e_-}\cap G_e$ has finite index in ${ G}_e$.
	\end{enumerate}
\end{defn}

\begin{defn}\label{defn:admissible}
	An extended admissible group $G$ is called an \emph{admissible group} if it has no vertex group of type $\mathcal{H}$.
\end{defn}

\begin{rem}
	The condition that  $G_v$ doesn't split relative to $\PP_v$  over a subgroup of an element of $\PP_v$ is  natural, as it  ensures that the decomposition $\cG$ of $G$ cannot be refined to a ``larger'' splitting of $G$.
\end{rem}

\begin{conv}
	For the rest of this paper, if $G$ is an extended admissible group, we will assume that all the data $\cG$, $G_v$, $Z_v$, $Q_v$, etc.\ in Definition~\ref{defn:admissible} are fixed, and will make use  of this notation without explanation.  If $G'$ is another extended admissible group, we use the notation  $\cG'$, $G'_v$, $Z'_v$, $Q'_v$ etc.
\end{conv}

\begin{rem}
	Croke--Kleiner defined a more restrictive notion of an admissible group, where they also assume each edge group $G_e$ is isomorphic to $\Z^2$ and each infinite cyclic  $Z_v\vartriangleleft G_v$ is central~\cite{CK02}. We say an admissible group is \emph{admissible in the sense of Croke--Kleiner} if it satisfies these additional constraints. If $G$ is an admissible group (as in Definition~\ref{defn:admissible}) and  vertex and edge groups are separable, then $G$ has a finite index subgroup that is admissible in the sense of Croke--Kleiner. The reason for working with the more general Definition~\ref{defn:admissible} is that it is more natural from the viewpoint of quasi-isometric rigidity,  in which groups that are abstractly commensurable are regarded as indistinguishable.
\end{rem}

Below are some examples of extended admissible groups.
\begin{exmp}\label{exmp:exampleCK}
	\begin{enumerate}
		\item (3-manifold groups) The fundamental group of a compact, orientable, irreducible 3-manifold $M$ with
		      empty or toroidal boundary is an extended admissible group. Seifert fibered and hyperbolic pieces correspond to type $\cS$ and $\cH$ vertex respectively. Fundamental groups of graph manifolds are admissible groups.

		\item (Torus complexes)
		      Let $n \ge 3$ be an integer. Let $T_1, T_2, \ldots, T_n$ be a family of flat two-dimensional tori. For each $i$, we choose a pair of simple closed geodesics $a_i$ and $b_i$ such that $\operatorname{length} (b_i) = \operatorname{length} (a_{i+1})$, identifying $b_i$ and $a_{i+1}$ and denote the resulting space by $X$.
		      The space $X$ is a graph of spaces with $n-1$ vertex spaces $V_{i} : = T_{i} \cup T_{i+1} / \{ b_i =  a_{i+1} \}$ (with $i \in \{1, \ldots, n-1\}$) and $n-2$ edge spaces $E_{i} : = V_{i} \cap V_{i+1}$.

		      The fundamental group $G = \pi_{1}(X)$ has a graph of groups structure where each vertex group is the fundamental group of the product of a figure eight and $S^1$. Vertex groups are isomorphic to  $F_{2} \times \Z$ and edge groups  are isomorphic to $\pi_1(E_i) \cong \Z^2$. The generators $[a_i], [b_i]$ of the edge group $\pi_1(E_i)$ each map to a generator of either a  $\Z$ or  $F_2$ factor of $F_{2} \times \Z$. It is clear that with this graph of groups structure, $\pi_1(X)$ is an admissible group.

	\end{enumerate}
\end{exmp}

\subsection{Properties of admissible groups}
We now prove some elementary facts concerning extended admissible groups.
\begin{lem}\label{lem:admissible_elem}
	Let $G$ be an extended admissible group and let $G_v$ be a type $\cS$ vertex.
	\begin{enumerate}
		\item The kernel $Z_v\lhd G_v$ is  unique up to commensurability.
		\item For each $e\in \lk(v)$,  $Z_v\leq G_e$.
		\item If $e$ is an edge with $v=e_-$ and $w=e_+$, then $Z_v\cap \tau_{\bar e}(Z_w)=\{1\}$.
	\end{enumerate}
\end{lem}
\begin{proof}
	(1): Suppose $Z_v,Z_v'\lhd G_v$ are infinite cyclic normal subgroups such that associated quotients  $Q_v$ and $Q'_v$ are non-elementary hyperbolic groups. Let $H_v$ be the kernel of the map
	\[
		G_v\to \Aut(Z_v)\times \Aut(Z'_v)
	\]
	induced by conjugation.
	Clearly $Z_v,Z_v'\leq Z(H_v)$,   $Z(H_v)/Z_v\leq Z(Q_v)$ and $Z(H_v)/Z'_v\leq Z(Q'_v)$. As $Q_v$ and $Q'_v$ are non-elementary hyperbolic groups, they have finite center. Thus $Z_v$ and $Z'_v$ are finite index subgroups of $Z(H_v)$, hence are commensurable.

	(2): Let  $H_v\leq G_v$ be the subgroup of index at most two centralizing $Z_v$. For each $g\in Z_v$, we have $H_v\cap G_e=g(H_v\cap G_e)g^{-1}$ is a subgroup of index at most two in both $G_e$ and $gG_eg^{-1}$. Thus $G_e$ and $gG_eg^{-1}$ are commensurable, hence $g\in G_e$.

	(3): This follows from the fact that $Z_v$ and $\tau_{\bar e}(Z_w)$ are infinite cyclic subgroups generating a finite index subgroup of a virtually  $\Z^2$ group.
\end{proof}

\begin{defn}[Kernels of vertex stabilizers]
	Let $\cG$ be the graph of groups associated to an extended  admissible group $G$, with associated Bass--Serre tree $T$. For $v\in T$, each vertex stabilizer $G_v$ is equal to some conjugate $gG_{\hat v}g^{-1}$ of  a vertex group $G_{\hat v}$ of $\cG$. If $G_{\hat v}$ is of type $\cS$, we  define the \emph{kernel} $Z_v\coloneqq gZ_{\hat v}g^{-1}$ of $G_v$ for all $v\in VT$, where $Z_{\hat v}\lhd G_{\hat v}$ is the kernel of the vertex group $G_{\hat v}$ as in Definition~\ref{defn:admissible}. By construction, we have $hZ_vh^{-1}=Z_{hv}$ for all $v\in VT$ and $h\in G$.
\end{defn}

\begin{lem}\label{lem:edge_vertex_int_admissible}
	Let  $G$ be an extended admissible group and let $T$ be the associated Bass--Serre tree.
	\begin{enumerate}
		\item Let $e,e'\in ET$ with $e\neq e'$. If $e,e'$ are incident to a type $\cS$ vertex $v$, then $G_e\cap G_{e'}$ contains $Z_v\cong \Z$ as a finite index subgroup. Otherwise, $G_e\cap G_{e'}$ is finite.
		\item Let $v,v'\in VT$ with $d_T(v,v')\geq 2$.
		      \begin{enumerate}
			      \item If $d_T(v,v')=2$ and the vertex $v''\in VT$ lying strictly between $v$ and $v'$ is of type $\cS$, then $G_v\cap G_{v'}$ contains $Z_{v''}$ as a finite index subgroup.
			      \item Otherwise,  $G_v\cap G_{v'}$ is finite.
		      \end{enumerate}
	\end{enumerate}
\end{lem}
\begin{proof}
	(1): Suppose $e,e'$ are incident to a type $\cS$ vertex $v$. It follows from Lemma~\ref{lem:admissible_elem} that $Z_v\leq G_e\cap G_{e'}$. By Definition~\ref{defn:admissible}, $G_e$ and $G_{e'}$ are not commensurable, hence  $G_{e}\cap G_{e'}$ is an infinite, infinite index subgroup of $G_e$, hence must be virtually cyclic, hence contains $Z_v$ as a finite index subgroup.

	Now suppose $e,e'$ are incident to a common type $\cH$ vertex $v$. Then by Definition~\ref{defn:admissible},  $G_e$ and $G_{e'}$ correspond to distinct peripheral subgroups of the relatively hyperbolic group $G_v$, hence have finite intersection~\cite{bowditch2012relatively}.

	If there is no common vertex incident with both $e$ and $e'$, consider an edge path $e_1,e_2,e_3$ on a geodesic from $e$ to $e'$, and set $v=(e_2)_-$ and $w=(e_2)_+$. In the case $v$  is of type $\cH$, we are done as $G_{e}\cap G_{e'}\leq G_{e_1}\cap G_{e_2}$ is finite. We argue similarly if $w$ is of type $\cH$, so we assume $v$ and $w$ are both of type $\cS$. Then  $Z_v$ is commensurable to $G_{e_1}\cap G_{e_2}$ and $Z_w$ is commensurable to $G_{e_2}\cap G_{e_3}$. Thus $G_e\cap G_{e'}\leq G_{e_1}\cap G_{e_2}\cap G_{e_3}$ is commensurable to a subgroup of $Z_v\cap Z_w$. As $Z_v,Z_w\leq G_{e_2}$ are infinite cyclic subgroups generating a finite index subgroup of $G_{e_2}$,  $Z_v \cap Z_w$ is finite. It follows $G_e\cap G_{e'}$ is also finite.

	(2): Set $e$ and $e'$ to be the first and last edges on a geodesic edge path from $v$ to $w$, and apply (1).
\end{proof}

We recall the following elementary lemma, which is a consequence of~\cite[Corollary 2.4]{moshersageevwhyte2011quasiactions} combined with the fact each coset $gH$ has finite Hausdorff distance from the subgroup $gHg^{-1}$.
\begin{lem}\label{lem:subgp_commensurability}
	Let $G$ be a finitely generated group with $g,h\in G$ and $H,K\leq G$ two subgroups. Then  $gHg^{-1}$ is commensurable to a subgroup of $hKh^{-1}$ if and only if $gH\subseteq N_r(hK)$ for some $r$ sufficiently large.
\end{lem}

We use this to deduce:
\begin{prop}\label{prop:coarse_int_vertex_edge}
	Let $(X,T)$ be the tree of spaces associated to an extended admissible group and let $a,b\in VT\sqcup ET$. The following are equivalent:
	\begin{enumerate}
		\item $X_a\subseteq N_r(X_b)$ for some  $r$;\label{item:coarse_int_vertex_edge1}
		\item $X_a\subseteq N_1(X_b)$;\label{item:coarse_int_vertex_edge2}
		\item Either $a=b$, or $a$ is an edge and $b\in \{a_-,a_+,\bar a\}$.\label{item:coarse_int_vertex_edge3}
	\end{enumerate}
	Moreover, $X_a$ and $X_b$ are at finite Hausdorff distance if and only if either $a=b$, or  $a$ and $b$ are edges with $\bar a=b$.
\end{prop}
\begin{proof}
	The directions (\ref{item:coarse_int_vertex_edge3})$\implies$ (\ref{item:coarse_int_vertex_edge2}) and (\ref{item:coarse_int_vertex_edge2})$\implies$ (\ref{item:coarse_int_vertex_edge1}) are clear.
	It follows from Lemma~\ref{lem:edge_vertex_int_admissible} that $G_a$ is commensurable to a  subgroup of $G_b$ if and only if either $a=b$, or $a$ is an edge and $b\in \{a_-,a_+,\bar a\}$. The equivalence of (3) and (1) now follows from Proposition~\ref{prop:treeofspacesBST} and Lemma~\ref{lem:subgp_commensurability}.
\end{proof}

If $G$ is a group, a set of subgroups  $\{H_i\}_{i\in I}$ is an \emph{almost malnormal family} if whenever there exist $i,j\in I$ and $g\in G$ such that $gH_ig^{-1}\cap H_j$ is infinite, $i=j$ and $g\in H_i$.  One source of relatively hyperbolic groups is the following:
\begin{thm}{\cite[Theorem 7.11]{bowditch2012relatively}}\label{thm:almalnorm_relhyp}
	If $G$ is a hyperbolic group and $\cH$ is an almost malnormal family of infinite  quasi-convex subgroups, then $G$ is hyperbolic relative to  $\cH$.
\end{thm}
We can use this to show:
\begin{prop}\label{prop:relhyp}
	Let  $G$ be an extended admissible group with associated graph of groups $\cG$. Let $G_v$ be a type $\cS$ vertex group of $\cG$  with kernel $Z_v$ and quotient $Q_v=G_v/Z_v$.  Then $Q_v$ is hyperbolic relative to \[\{G_e/Z_v\mid e\in \lk(v)\}.\]
\end{prop}
\begin{proof}
	Lemma~\ref{lem:admissible_elem} ensures that $Z_v\leq G_e$ for each $e\in \lk(v)$, so the above expression makes sense.
	For each $e\in \lk(v)$, set $H_e\coloneqq G_e/Z_v$.
	Since $G_e$ is virtually $\Z^2$ and $Z_v\cong \Z$, each $H_e$ is virtually infinite cyclic. Thus $H_e$ is a quasi-convex subgroup of $Q_v$ (see~\cite[Lemma~3.6, Lemma~3.10 Chapter III.$\Gamma$]{BH99}).

	Let $g\in G_v$ and $e,e'\in \lk(v)$. Set $\bar g=gZ_v$. If $\bar g H_e \bar g^{-1}\cap H_{e'}$ is infinite, then as $H_e$ and $H_{e'}$ are infinite cyclic, $\bar g H_e \bar g^{-1}$ and $H_{e'}$ are commensurable. Therefore, $gG_eg^{-1}$ and $G_{e'}$ are commensurable, hence by  Definition~\ref{defn:admissible}, $g\in G_e$ and $e=e'$. Thus $\{H_e\mid e\in\lk(v)\}$ is an almost malnormal quasi-convex collection of subgroups. The result now follows from Theorem~\ref{thm:almalnorm_relhyp}.
\end{proof}

\section{Quasi-isometric rigidity of admissible groups}\label{sec:qirigidity}

In this section, we restrict our attention to admissible groups as defined in Definition~\ref{defn:admissible}, and prove special cases of Theorems~\ref{thm:CKrigidity_qi} and~\ref{thm:CKrigidity_graphofgroups} for this class of groups.

\subsection{The geometry of vertex and edge spaces}\label{sub:geovertexedge}
We first discuss some properties concerning the geometry of vertex and edge spaces of admissible groups.
For the remainder of this subsection, we fix an admissible group $G$, with associated graph of groups $\cG$ and tree of spaces $(X,T)$.

We now define auxiliary data associated to each vertex space of $X$.
Recall from the construction in Section~\ref{sec:TofS} that each vertex space $X_v$ of $X$ is identified with the Cayley graph of a vertex group $G_{\hat v}$ of $\cG$ with respect to some generating set $S_{\hat v}$. Furthermore, as each vertex group  $G_{\hat v}$ is of type $\cS$, it has an infinite cyclic kernel $Z_{\hat v}\vartriangleleft G_{\hat v}$.
Let $q_{\hat v}:G_{\hat v}\to Q_{\hat v}$ be the quotient map.

\begin{defn}\label{defn:quotient}
	Let $X_v$ be a vertex space of $X$, which we  identify with the Cayley graph of some vertex group $G_{\hat v}$. With $Q_{\hat v}$, $S_{\hat v}$, $Z_{\hat v}$ as above, we have the following:
	\begin{enumerate}
		\item The \emph{quotient space} of $X_v$ is a copy $Y_v$ of the Cayley graph of $Q_{\hat v}$ with respect to the generating set $\{q_{\hat v}(s)\mid s\in S_{\hat v}\}$.
		\item The \emph{quotient map} $\pi_v:X_v\to Y_v$ is the graph morphism taking the edge $(g,gs)$ in $X_v$ to the edge $(q_{\hat v}(g),q_{\hat v}(g)q_{\hat v}(s))$ in $Y_v$.
		\item For each $e\in ET$ with $v=e_-$, we define $\ell_e\coloneqq \pi_v(X_e)\subseteq Y_v$.
	\end{enumerate}
\end{defn}

We now discuss some properties of the spaces and maps defined in Definition~\ref{defn:quotient}. The following is evident from the definitions.
\begin{lem}\label{lem:quotient map}
	For each $v\in VT$, the following hold.
	\begin{enumerate}
		\item The map $\pi_v:X_v\to Y_v$ is $1$-Lipschitz.
		\item For each $x\in X_v$ and $y\in Y_v$ with $d_{Y_v}(\pi_v(x),y)=R$, there is some $\tilde{y}\in\pi_v^{-1}(y)$ with $d_{X_v}(x,\tilde{y})=R$.
	\end{enumerate}
\end{lem}

By the construction of $X$ in Section~\ref{sec:TofS}, under the identification of $X_v$ with the Cayley graph of a vertex group $G_{\hat v}$, the edge spaces $X_e$ with $e\in\lk(v)$ are identified with left cosets of edge groups $G_{\hat e}$ such that $G_{\hat e}\leq G_{\hat v}$. Therefore, we deduce:
\begin{lem}
	For each vertex $v\in VT$, the set
	\[\{\ell_e=\pi_v(X_e)\subseteq Y_v\mid e\in \lk v\}\] is identified with the set of left cosets of $\{q_{\hat v}(G_{\hat e})\mid \hat e\in \lk(\hat v)\}$ in $Q_{\hat v}$.
\end{lem}
Combined with Proposition~\ref{prop:relhyp}, we thus deduce that:
\begin{cor}\label{cor:cay_relhyp}
	Each $Y_v$ is the Cayley graph of a relatively hyperbolic group and the set $\{\ell_e\mid e\in \lk(v)\}$ is precisely the set of left cosets of the peripheral subgroups, which are all 2-ended.
\end{cor}

This tells us a lot about the geometry of $Y_v$ and $\{\ell_e\mid e\in \lk(v)\}$. The following is a straightforward consequence of the characterization of relative hyperbolicity given by Sisto~\cite[Definition 2.1, Lemma 2.3 and Theorem 2.14]{Sis13}, coupled with the fact there are only finitely many isometry types of $Y_v$.
\begin{cor}[{\cite{Sis13}}]\label{cor:proj_bound}
	There is a uniform constant $B$ such that following holds. For every $v\in VT$ and $e\in \lk(v)$, let  $\proj_{\ell_e}:Y_v\to \ell_e$ be a closest point projection. Then:
	\begin{enumerate}
		\item for all $e'\in \lk(v)\setminus\{e\}$, $\diam(\proj_{\ell_e}(\ell_{e'}))\leq B$.
		\item for all $R\geq 0$ and $Z\subseteq Y_v$, $N_R(\ell_e)\cap Z\subseteq N_{R+B}(\proj_{\ell_e}(Z))$.
		\item each $\proj_{\ell_e}$ is $(1,B)$-Lipschitz.
	\end{enumerate}
\end{cor}

We now describe the structure of vertex spaces of $X$. We recall the following result of Gersten~\cite{gersten1992bounded}; see also~\cite[\S 11.19]{DK18}.

\begin{thm}\label{thm:qi_to_product}
	Let $G$ be a group fitting into the short exact sequence
	\[
		1\to \Z\to G\to Q\to 1,
	\]
	where $Q$ is a non-elementary hyperbolic group. Then there is a quasi-isometry $f:G\to \E\times Q$ such that the composition of $f$ with the projection $\E\times Q\to Q$ agrees with the quotient map $G\to Q$.
\end{thm}
Since each vertex space $X_v$ of $X$ is isomorphic to the Cayley graph of a vertex group of $\cG$, and the quotient map $X_v\to Y_v$ is the projection to a Cayley graph of the quotient we conclude:
\begin{cor}\label{cor:vspace_qitoprod}
	There exist $K\geq 1$ and $A\geq 0$ such that for each vertex $v\in VT$, there is a $(K,A)$-quasi-isometry $f_v:X_v\to \E\times Y_v$, such that the composition of $f_v$ with the projection to $Y_v$ coincides with the quotient map $\pi_v$.
\end{cor}

Let us also recall that every 2-ended subgroup of a hyperbolic group is quasi-isometrically embedded. Thus every inclusion $\ell_e\to Y_v$ is a quasi-isometric embedding, hence so is every inclusion $\E\times \ell_e\to \E\times Y_v$. Since the map $f_v$ in Corollary~\ref{cor:vspace_qitoprod} maps $X_e$ to $\E\times \ell_e$ up to uniform Hausdorff distance, we conclude:
\begin{cor}\label{cor:espace_qi_embedded_in_vspace}
	There exist $K\geq 1$ and $A\geq 0$ such that for each vertex $e\in VT$ with $e_-=v$, the inclusion $X_e\to X_v$ is a $(K,A)$-quasi-isometric embedding.
\end{cor}

We also have the following useful formula for the distance between elements of an edge group in terms of projections to the hyperbolic quotients of adjacent vertex groups.

\begin{lem}\label{lem:dist_projs}
	There exist $K\geq 1$ and $A\geq 0$ such that for every $e\in ET$, setting $v=e_-$ and $w=e_+$, we have  \begin{align*}
		\frac{1}{K}d_{X_v}(x,y)-A \leq d_{Y_v}(\pi_v(x),\pi_v(y))+d_{Y_{w}}(\pi_w(\alpha_e(x)),\pi_w(\alpha_e(y)))\leq Kd_{X_v}(x,y)+A
	\end{align*}
	for all $x,y\in X_e$.
\end{lem}

\begin{proof}
	Via the construction of edge and vertex spaces in Section~\ref{sec:TofS}, it is enough to show the corresponding result for edge groups of $\cG$. More precisely, we will show that if $G_e$ is an edge group of $\cG$ with $e_-=v$ and $e_+=w$, there exist constants $K\geq 1$ and $A\geq 0$ such that for all $g,k\in G_e$ \[\frac{1}{K}d_{G_v}(g,k)-A \leq d_{Q_v}(q_v(g),q_v(k))+d_{Q_w}(q_w(\tau_e(g)),q_w(\tau_e(k)))\leq Kd_{G_v}(g,k)+A,\] where $q_v:G_v\to Q_v$ and $q_w:G_w\to Q_w$ are quotient maps and $\tau_e:G_e\to G_w$ is the edge map.

	We pick generators $a$ and $b$ of the infinite cyclic subgroups $Z_v\leq G_e$ and $\tau_{\bar e}(Z_w)$.
	By Lemma~\ref{lem:admissible_elem}, $Z_v\cap\tau_{\bar e}(Z_w)=\{1\}$.
	Since $\ker(q_v)=\langle a\rangle$ and $\ker (q_w\circ \tau_e)=\langle b\rangle$, it follows that $\bar b\coloneqq q_v(b)$ and $\bar a\coloneqq q_w(\tau_e(a))$ are infinite-order elements of $Q_v$ and $Q_w$ respectively.
	By Definition~\ref{defn:admissible}, $a$ and $b$ generate a finite index subgroup of the edge group $G_e$.  As $G_e$ is virtually $\Z^2$, after replacing $a$ and $b$ with powers if needed, we deduce $a,b$ generate a finite index subgroup $H$ of $G_e$ isomorphic to $\Z^2$.

	Let $d_H$ be the word metric on   $H=\langle a,b\rangle\cong \Z^2$ with respect to $\{a,b\}$.
	Let $g,k\in H$ and suppose $g^{-1}k=a^ib^j$ Then $d_H(g,k)=|i|+ |j|$. Now we have \[
		d_{Q_v}(q_v(g),q_v(k))=d_{Q_v}(1,q_v(a^ib^j))=d_{Q_v}(1,\bar b^j)
	\]
	and similarly
	\[
		d_{Q_w}(q_w(\tau_e(g)),q_w(\tau_e(k)))=d_{Q_w}(1,q_w(\tau_e(a^ib^j)))=d_{Q_w}(1,\bar a^i).
	\]
	As $\bar b$ and $\bar a$ are infinite order elements  of $Q_v$ and $Q_w$,  the maps $j\mapsto \bar b^j$ and $i\mapsto \bar b^i$ are quasi-isometric embeddings. Since $H$ is finite index in $G_e$ and by Corollary~\ref{cor:espace_qi_embedded_in_vspace},  the inclusion $H\to G_e\to G_v$ is a quasi-isometric embedding, there exist constants $K$ and $A$ such that
	\[\frac{1}{K}d_{G_v}(g,k)-A\leq d_{Q_v}(q_v(g),q_v(k))+d_{Q_w}(q_w(\tau_e(g)),q_w(\tau_e(k)))\leq Kd_{G_v}(g,k)+A\] for all $g,k\in H$. As $H$ is a finite index subgroup of $G_e$ and $q_v$ and $q_w\circ \tau_e$ are coarse Lipschitz, the above inequality holds for all $g,k\in G_e$ after increasing $K$ and $A$.
\end{proof}

\subsection{Vertex and edge spaces are quasi-isometrically embedded}
The main result of this subsection is the following:

\begin{thm}\label{thm:edge_qi_embedded}
	Let $X$ be the tree of spaces associated to an admissible group. Then edge spaces of $X$ are quasi-isometrically embedded in $X$.
\end{thm}
Suppose $X$ is a tree of spaces associated to a finite graph of finitely generated groups $\cG$, with $G=\pi_1(\cG)$. Proposition~\ref{prop:treeofspacesBST} easily implies that vertex (resp.\ edge) spaces of $X$ are quasi-isometrically embedded in $X$ if and only if vertex (resp.\ edge) groups of $\cG$ are quasi-isometrically embedded in $G$. Since there are only finitely many $G$-orbits of vertex and edge spaces of $X$, if all vertex (resp.\ edge) spaces of $X$ are quasi-isometrically embedded, there exist $K\geq 1$ and $A\geq 0$ such that every vertex (resp.\ edge) space is  $(K,A)$-quasi-isometrically embedded.
Combining Theorem~\ref{thm:edge_qi_embedded} with these observations  and  Lemma~\ref{lem:vertex_edge_QIembeddings}, we deduce:
\begin{cor}\label{cor:ve_spaces_qi_embedded}
	Let $X$ be a tree of spaces associated to an admissible group. Then there exist constants  $K\geq 1$ and $A\geq 0$ such that every vertex and edge space of $X$ is  $(K,A)$-quasi-isometrically embedded in $X$.
\end{cor}

It remains to prove Theorem~\ref{thm:edge_qi_embedded}, which we do using an argument similar to that used in~\cite[\S 7]{frigeriolafontsistosro2015rigidity}.

\begin{rem}
	When restricting to admissible groups in the sense of Croke--Kleiner~\cite{CK02}, Theorem~\ref{thm:edge_qi_embedded} can be deduced by combining the main results of~\cite{HRSS22} and~\cite{HHP20} with  the fact that finitely generated abelian subgroups of semi-hyperbolic groups are quasi-isometrically embedded. It is likely that the proof in~\cite{HRSS22} holds verbatim for the  more general class of admissible groups under consideration here. However, we present a more elementary and self-contained proof of Theorem~\ref{thm:edge_qi_embedded} instead.
\end{rem}

We make use of the following lemma, which is a variation of a result of Osin; see also~\cite[Proposition 7.4]{frigeriolafontsistosro2015rigidity}.
\begin{lem}[{\cite[Lemma 3.2]{osin2006Relatively}}]\label{lem:dist_relhyp_bound}
	Let $G$ be a finitely generated group that is hyperbolic relative to $\mathcal{H} = \{H_1,\dots, H_n\}$, equipped with the word metric $d$ with respect to a finite generating set $S$.  Let $\cP$ be the set of left cosets of elements of $\cH$. There is a constant $M$ such that the following holds.

	Suppose there exist $\gamma_0^\pm,\dots, \gamma_n^\pm\in G$ and distinct $P_0,\dots, P_n\in \cP$ such that for all $i$,
	$\gamma_i^-\in P_i$ and $\gamma_i^+\in P_{i+1}$ (with $\gamma_n^+\in P_0$). Then \[\sum_{i=1}^n d_S(\gamma_{i-1}^+,\gamma_{i}^-)\leq M \sum_{i=0}^nd_S(\gamma_i^-,\gamma_i^+).\]
\end{lem}
\begin{proof}
	Up to modifying $M$, the conclusion is independent of the generating set chosen. We can therefore choose a finite generating set $S$ of $G$ and constant $N>1$ such that~\cite[Lemma 3.2]{osin2006Relatively} holds, where $N=ML$ in the notation of the cited lemma. Throughout the subsequent argument, we assume indices are taken modulo $n+1$.

	Pick $j_i$ such that $P_i$ is a left coset of $H_{j_i}$.
	Let $w_i$ be a minimal length word in $S$ representing $(\gamma_{i}^-)^{-1}\gamma_{i}^+$, and let  $h_i=(\gamma_{i-1}^+)^{-1}\gamma_i^-$. Since  $\gamma_{i-1}^+$ and $\gamma_i^{-}$  both lie in $P_i$, we see that $h_i\in H_{j_i}$. Therefore $h_0w_0h_1w_1\dots h_n w_n$ represents the identity element of $G$.

	We write the word $w_i$ as $u_iw'_iv_i$, where $u_i$ and $v_i$ are the initial and terminal subwords of $w_i$ of maximal length such that $u_i\in H_{j_i}$ and $v_i\in H_{j_{i+1}}$. We set $h'_i=v_{i-1}h_iu_i\in H_{j_i}$. Therefore,
	\[
		\ell=h'_0 w'_0 h'_1 \dots h_n'w'_n
	\]
	is a word in $S\cup \cH\coloneqq S\cup \bigcup_{i=1}^nH_i$ representing the identity in $G$. In other words, $\ell$ corresponds to a loop of length $n+1+\sum_{i=0}^n|w'_i|_S$  in the Cayley graph of $G$ with respect to $S\cup \cH$.

	By construction, the segments of the loop labeled by $h'_i$ are $\cH$-isolated components in the sense of~\cite{osin2006Relatively}. Thus~\cite[Lemma 3.2]{osin2006Relatively} implies \[
		\sum_{i=0}^n|h'_i|_S\leq N\left(n+1+\sum_{i=0}^n|w'_i|_S\right)
	\] Since $P_i\neq P_{i+1}$, each $w_i$ is not the trivial word and so $n+1\leq\sum_{i=0}^n|w_i|_S$.  Putting everything together, we have
	\begin{align*}
		\sum_{i=1}^n d_S(\gamma_{i-1}^+,\gamma_{i}^-)
		 & =\sum_{i=1}^n |h_i|_S\leq \sum_{i=0}^n \left(|h'_i|_S+|u_i|_S+|v_i|_S\right)           \\
		 & \leq N\left(n+1+\sum_{i=0}^n|w'_i|_S\right)+N\sum_{i=0}^n \left(|u_i|_S+|v_i|_S\right) \\
		 & \leq N\left(n+1+\sum_{i=0}^n|w_i|_S\right)\leq 2N\sum_{i=0}^n|w_i|_S.\qedhere
	\end{align*}
\end{proof}

The following argument is similar to that used in~\cite[\S 7]{frigeriolafontsistosro2015rigidity}, and comprises the key step in proving Theorem~\ref{thm:edge_qi_embedded}.

\begin{lem}\label{lem:path_upperbound}
	There is a constant $E$ such that for all $e$ and all $x,y\in X_e$, if $e_-=v$ and $\gamma$ is a continuous path in $X$ connecting $x$ to $y$ and disjoint from $X_{\bar e}$, then:
	\begin{enumerate}
		\item  $d_{Y_v}(\pi_v(x),\pi_v(y))\leq E\length (\gamma)+E$.\label{item:path_upperbound_proj}
		\item $d_{X_v}(x,y)\leq E\length (\gamma)+E$.\label{item:path_upperbound_total}
	\end{enumerate}
\end{lem}

\begin{proof}
	Throughout the proof, we use the notation that if $p$ is a path from $a$ and $b$, then $a$ and $b$ are denoted by $p^-$ and $p^+$ respectively. We fix constants $K$, $A$, $B$ and $M$ such that Corollary~\ref{cor:proj_bound}, Lemma~\ref{lem:dist_projs} and Lemma~\ref{lem:dist_relhyp_bound} hold.  Since each edge stabilizer $G_e$ is virtually $\Z^2$ and acts co-compactly on the corresponding edge space $X_e$, we can choose a finite index abelian subgroup  $H_e\leq G_e$ and a constant $C$ such that $N_C^{X_{e}}(H_e\cdot x)=X_e$ for all $x\in X_e$. Moreover, since there are only finitely many $G$ orbits of edge spaces, $C$ can be chosen independently of $e$.

	It is sufficient to prove the statement in the case $x$ and $y$ are vertices and  $\gamma$ is a  combinatorial path satisfying the hypothesis of the lemma.
	We can write $\gamma$ as a  concatenation
	\[\gamma=\gamma_0\cdot \alpha_1\cdot \delta_1\cdot \beta_1\cdot \gamma_1\cdot \alpha_2\cdot \delta_2\cdot \beta_2\cdot \gamma_2 \dots \gamma_n,\] where:
	\begin{itemize}
		\item each $\gamma_i$ is a path in $X_{v}$ from a vertex of $X_{e_{i}}$ to a vertex of $X_{e_{i+1}}$ (we assume $e_0=e_{n+1}=e$).
		\item each $\delta_i$ is a path  with endpoints in $X_{\bar{e_i}}$ and disjoint from $X_{e_i}$.
		\item each $\alpha_i$ (resp $\beta_i$) is a unit length interval from $X_{e_i}$ to $X_{\bar{e_i}}$ (resp $X_{\bar{e_i}}$ to $X_{e_i}$). This implies $\alpha_{e_i}(\alpha_i^-)=\alpha_i^+$ and $\alpha_{\bar e_i}(\beta_i^-)=\beta_i^+$.
	\end{itemize} See Figure~\ref{fig:path} for a  diagram of such a  path $\gamma$.
	\begin{figure}[ht]
		\centering
		\incfig{path}
		\caption{A  diagram of the  path $\gamma$.}\label{fig:path}
	\end{figure}

	We remark that $n\leq \length(\gamma)$ since each $\alpha_i$ has unit length.
	Note also that $e_i\neq e$ for $1\leq i\leq n$, by the assumption $\gamma$ is disjoint from $X_{\bar e}$.

	We claim that, after possibly replacing $\gamma$ with a path of length at most $\ell(\gamma)(4C+1)$ if necessary, we may assume that the sequence of edges $e_1,\dots, e_n$ contains no repetitions.  Indeed, suppose that  $e_i=e_j$ for $1\leq i<j\leq n$.  By the choice of $C$, we can pick $a,b\in H_{\bar e_i}$ such that $d_{X_{\bar e_i}}(a\delta_j^-,\delta_i^+)\leq C$ and $d_{X_{\bar e_i}}(b\delta_i^+,a\delta_j^+)\leq C$. Note that
	\begin{align*}
		d_{X_{\bar e_i}}(b\delta_j^-,\delta_j^+)=d_{X_{\bar e_i}}(ab\delta_j^-,a\delta_j^+)=d_X(ba\delta_j^-,a\delta_j^+)\leq d_{X_{\bar e_i}}(ba\delta_j^-,b\delta_i^+)+d_{X_{\bar e_i}}(b\delta_i^+,a\delta_j^+)\leq 2C
	\end{align*} since $a,b\in H_{\bar e_i}\cong \Z^2$ commute.
	We now consider the concatenation
	\[\gamma'=[\gamma^-,\delta_i^+]_\gamma\cdot[\delta_i^+,a\delta_j^-]_{X_{\bar e_i}}\cdot  [a\delta_j^-,a\delta_j^+]_{a\gamma}\cdot [a\delta_j^+,b\delta_i^+]_{X_{\bar e_i}}\cdot [b\delta_i^+,b\delta_{j}^-]_{b\gamma}\cdot [b\delta_{j}^-,\delta_j^+]_{X_{\bar e_i}}\cdot [\delta_j^+,\gamma^+]_\gamma\]
	which satisfies $\ell(\gamma')\leq  \ell(\gamma)+4C$. Replacing $\gamma$ with $\gamma'$ thus has the effect of removing a repeated edge from the sequence  $e_1,\dots, e_n$ whilst increasing the length of $\gamma$ by at most $4C$.  Repeating this procedure at most $n\leq \ell(\gamma)$ times proves the claim.

	For $0\leq i\leq n$, set  $\epsilon_i\coloneqq   \pi_v\circ \gamma_i$.

	We now prove (\ref{item:path_upperbound_proj}). Set $p_i^\pm=\proj_{\ell_e}(\epsilon^\pm_i)$. For each $i$, we  have $d_{Y_v}(p_i^-,p_i^+)\leq \length(\gamma_i)+B$   via the coarse Lipschitz constants of $\proj_{\ell_e}$ and $\pi_v$ in Lemmas~\ref{lem:quotient map} and Corollary~\ref{cor:proj_bound}.
	We also have $d_{Y_v}(p_{i-1}^+,p_i^-)\leq B$ since  $\proj_{\ell_e}(\ell_{e_i})$ has diameter at most $B$.
	Thus \[d_{Y_v}(p_0^-,p_n^+)\leq \sum_{i=0}^nd_{Y_v}(p_i^-,p_i^+)+nB\leq  \sum_{i=0}^n\length (\gamma_i) +(2n+1)B.\]
	Since $n\leq \length (\gamma)$, we have
	\[
		d_{Y_v}(\pi_v(x),\pi_v(y))=d_{Y_v}(p_0^-,p_n^+)\leq (2B+1)\length(\ell)+B
	\] completing the proof of (\ref{item:path_upperbound_proj}).

	We now fix $E$ satisfying the conclusion of  (\ref{item:path_upperbound_proj}), and  prove (\ref{item:path_upperbound_total}).  Set $I=\{1,\dots,n\}$ and \[J=\{i\in I \mid 2K d_{Y_v}(\epsilon_{i-1}^+,\epsilon_i^-) \leq d_{X_v}(\gamma_{i-1}^+,\gamma_i^-)\}.\]
	The triangle inequality yields the expression
	\[d_{X_v}(x,y)=d_{X_v}(\gamma_0^-,\gamma_n^+)\leq \sum_{i=0}^nd_{X_v}(\gamma_i^-,\gamma_i^+)+\sum_{i\in J}d_{X_v}(\gamma_{i-1}^+,\gamma_i^-) +\sum_{i\in I\setminus J}d_{X_v}(\gamma_{i-1}^+,\gamma_i^-). \] Therefore, at least one of the three terms in the above inequality is at least $\frac{d_{X_v}(x,y)}{3}$. We consider the three cases separately and prove (\ref{item:path_upperbound_total}) holds in each case.

	{Case 1: $d_{X_v}(x,y) / 3 \leq \sum_{i=0}^nd_{X_v}(\gamma_i^-,\gamma_i^+)$.}

	This is immediate, since \[d_{X_v}(x,y)\leq 3\sum_{i=0}^nd_{X_v}(\gamma_i^-,\gamma_i^+)\leq 3\sum_{i=0}^n\length(\gamma_i)\leq 3\length(\gamma).\]

	{Case 2: $d_{X_v}(x,y) / 3 \leq \sum_{i\in J}d_{X_v}(\gamma_{i-1}^+,\gamma_i^-)$.}

	Set $\hat \epsilon_i=\pi_{v_i}\circ \delta_i$.
	Recall $\delta_i$ is an edge path with endpoints in $X_{\bar e_i}$ disjoint from $X_{e_i}$. Hence, we can apply (\ref{item:path_upperbound_proj}) to $\delta_i$ to give
	\[d_{Y_{v_i}}(\hat \epsilon_{i}^-,\hat \epsilon_i^+)\leq E\length(\delta_i)+E\]

	If $i\in J$, then Lemma~\ref{lem:dist_projs} and the definition of $J$ yields
	\begin{align*}
		\frac{1}{K}d_{X_v}(\gamma_{i-1}^+,\gamma_{i}^-)-A & \leq d_{Y_v}(\epsilon_{i-1}^+,\epsilon_i^-)+d_{Y_{v_i}}(\hat \epsilon_{i}^-,\hat \epsilon_i^+)
		\\~\\
		                                                  & \leq \frac{1}{2K}d_{X_v}(\gamma_{i-1}^+,\gamma_i^-)+ d_{Y_{v_i}}(\hat \epsilon_{i}^-,\hat \epsilon_i^+)
	\end{align*}
	since $\pi_{v_i}(\alpha_e(\gamma_{i-1}^+))=\hat \epsilon_{i}^-$ and $\pi_{v_i}(\alpha_e(\gamma_{i}^-))=\hat \epsilon_i^+$.

	Consequently,  $d_{X_v}(\gamma_{i-1}^+,\gamma_{i}^-)\leq 2Kd_{Y_{v_i}}(\hat \epsilon_{i}^-,\hat \epsilon_i^+)+2KA$ for $i\in J$.
	Thus
	\begin{align*}
		d_{X_v}(x,y) & \leq 3\sum_{i\in J} d_{X_v}(\gamma_{i-1}^+,\gamma_i^-)\leq\sum_{i\in J} \left(6Kd_{Y_{v_i}}(\hat \epsilon_{i}^-,\hat \epsilon_i^+)+6KA\right) \\
		             & \leq \sum_{i =1}^n 6K E\length(\delta_i)+6Kn(E+A)                                                                                             \\
		             & \leq   6K(2E+A)\length(\gamma)
	\end{align*}
	and we are done.

	{Case 3: $d_{X_v}(x,y) / 3 \leq\sum_{i\in I\setminus J}d_{X_v}(\gamma_{i-1}^+,\gamma_i^-)$.}

	Combining Lemma~\ref{lem:dist_relhyp_bound}, the definition of $J$, and the fact that  $\pi_v$ is $1$-Lipschitz, we deduce
	\begin{align*}
		d_{X_v}(x,y) & \leq 3\sum_{i\in I\setminus J}d_{X_v}(\gamma_{i-1}^+,\gamma_i^-)< 6K\sum_{i\in I\setminus J}d_{Y_v}(\epsilon_{i-1}^+,\epsilon_i^-)\leq 6K\sum_{i=1 }^nd_{Y_v}(\epsilon_{i-1}^+,\epsilon_i^-) \\
		             & \leq 6KM\sum_{i=0 }^nd_{Y_v}(\epsilon_{i}^-,\epsilon_i^+)\leq 6KM\sum_{i=0 }^n\length(\gamma_i)\leq 6KM\length(\gamma).
	\end{align*}
	This concludes the proof in all cases.
\end{proof}

\begin{proof}[Proof of Theorem~\ref{thm:edge_qi_embedded}]
	Let $E$ be the constant as in Lemma~\ref{lem:path_upperbound}. Let $e_-=v$ and $e_+=w$. By Corollary~\ref{cor:espace_qi_embedded_in_vspace}, we can pick $K$ and $A$ such that $X_e\to X_v$ and $X_{\bar e}\to X_w$ are $(K,A)$-quasi-isometric embeddings.

	Let $x,y\in X_e$ and let $\gamma$  be a geodesic path from $x$ to $y$ in $X$, which is necessarily an edge path.
	We can write $\gamma$ as a  concatenation
	\[\gamma=\gamma_0\cdot \alpha_1\cdot \gamma_1\cdot \alpha_2\cdot \dots \alpha_n\cdot \gamma_n\]
	such that the following hold:
	\begin{itemize}
		\item If $i$ is even (resp.\ odd),  $\gamma_i$ is  disjoint from $X_{\bar e}$ (resp.\ $X_e$) with endpoints on  $X_e$ (resp.\ $X_{\bar e}$).
		\item If $i$ is odd (resp.\ even), $\alpha_i$ is  a unit  interval from  $X_{e}$ (resp.\ $X_{\bar{e}}$) to $X_{\bar{e}}$ (resp.\ $X_e$).
	\end{itemize}
	In particular, we note that as the endpoints of $\gamma$ are in $X_e$, $n$ must be even. (However, $\gamma_0$ or $\gamma_n$ might have length zero.) Since $\alpha_i$ has length $1$, we see $n\leq \length (\gamma)$. Moreover, Lemma~\ref{lem:path_upperbound} applies to each $\gamma_i$.

	For odd $i$,  Lemma~\ref{lem:path_upperbound} implies $d_{X_w}(\gamma_i^-,\gamma_i^+)\leq E\length(\gamma_i)+E$. Thus $d_{X_{\bar e}}(\gamma_i^-,\gamma_i^+)\leq KE\length(\gamma_i)+KE+KA$.
	As $\alpha_{\bar e}:X_{\bar e}\to X_e$ is an isometric embedding and $\alpha_{\bar e}(\gamma_i^-)=\gamma_{i-1}^+$ and $\alpha_{\bar e}(\gamma_i^+)=\gamma_{i+1}^-$, we deduce that $d_{X_e}(\gamma_{i-1}^+,\gamma_{i+1}^-)\leq KE\length(\gamma_i)+KE+KA$.
	For even $i$, an argument identical to that given above implies  $d_{X_e}(\gamma_i^-,\gamma_i^+)\leq KE\length(\gamma_i)+KE+KA$.

	Putting everything together and applying the triangle inequality we conclude that
	\begin{align*}d_{X_e}(x,y) & =d_{X_e}(\gamma_0^-,\gamma_n^+)\leq\sum_{i=0}^{\frac{n}{2}}d_{X_e}(\gamma_{2i}^-,\gamma_{2i}^+)+\sum_{i=1}^{\frac{n}{2}}d_{X_e}(\gamma_{2i-2}^+,\gamma_{2i}^-) \\
                           & \leq KE\sum_{i=0}^n\length(\gamma_i)+(n+1)(KE+KA)                                                                                                              \\
                           & \leq (2KE+KA)\length(\gamma)+KE                                                                                                                                \\
                           & =(2KE+KA)d_X(x,y)+KE.
	\end{align*}
	Clearly $d_X(x,y)\leq d_{X_e}(x,y)$, since $X_e$ and $X$ are geodesic metric spaces and $X_e$ is a subspace of $X$. Thus the inclusion $X_e\to X$ is a quasi-isometric embedding.
\end{proof}

\subsection{bi-Lipschitz maps of the asymptotic cone}

We now begin our study of the asymptotic cone of the tree of spaces of an admissible group. We fix a tree of spaces $(X,T)$ associated with an admissible group, with associated Bass--Serre tree $T$, and fix an asymptotic cone $X_\omega=X_\omega((b_i),(\lambda_i))$ of $X$.

\begin{defn}\label{defn:omegavertexedgespaces}
	Let $\cV$ and $\cE$  be the set of all vertex spaces and edge spaces of $X$ respectively.
	We define an  \emph{$\omega$-vertex space} and \emph{$\omega$-edge space} to be an element of $\cV_\omega$ and $\cE_\omega$ respectively, as in Definition~\ref{defn:ultralimit}.
\end{defn}

The following lemma describes the structure of $\omega$-vertex spaces.
\begin{prop}\label{prop:ultralimit_vertex_space}
	Let $\lim_\omega X_{v_i}\in \cV_\omega$ and consider \[\cA_\omega\coloneqq \{\lim_\omega X_{e_i}\in \cE_\omega\mid e_i\in \lk(v_i) \text{ for all $i$}\}.\] Fix $[(a_i)]\in \lim_\omega X_{v_i}$ with $a_i\in X_{v_i}$ for all $i$, and set
	\[Y_\omega\coloneqq \lim_\omega\left(Y_{v_i},\pi_{v_i}(a_i),\frac{d_{Y_{v_i}}}{\lambda_i}\right).\] Let $\pi_\omega:\lim_\omega X_{v_i}\to Y_\omega$ be the map $\lim_\omega (\pi_{v_i})$ and set
	\[\cL_\omega\coloneqq \{\pi_\omega(\lim_\omega X_{e_i})=\lim_\omega \ell_{e_i}\mid \lim_\omega X_{e_i}\in \cA_\omega \}.\]
	Then there is a bi-Lipschitz equivalence $f_\omega:\lim_\omega X_{v_i}\to \E\times Y_\omega$ such that the following hold:
	\begin{enumerate}
		\item\label{prop:ultralimit_vertex_space1} $Y_\omega$ is  a  geodesically complete $\R$-tree that branches everywhere.
		\item\label{prop:ultralimit_vertex_space2} Every  $\ell_\omega=\lim_\omega( \ell_{e_i})\in \cL_\omega$  is a bi-infinite geodesic line  in $Y_\omega$ and $f_\omega(\lim_{\omega}X_{e_i})=\E\times \ell_{\omega}$.
		\item\label{prop:ultralimit_vertex_space3}$Y_\omega$ is tree-graded with respect to $\cL_\omega$.
	\end{enumerate}
\end{prop}

\begin{proof}
	For each $i$, we consider the  commutative diagram
	\[\begin{tikzcd}
			{X_{v_i}} & {\mathbb E\times Y_{v_i}} \\
			{Y_{v_i}} & {Y_{v_i}}
			\arrow["f_{v_i}", from=1-1, to=1-2]
			\arrow["\pi_{v_i}"', from=1-1, to=2-1]
			\arrow["\text{Id}"', from=2-1, to=2-2]
			\arrow["q_{v_i}", from=1-2, to=2-2]
		\end{tikzcd}\]
	where  $q_{v_i}$ is a projection and $f_{v_i}$ is a quasi-isometry as in Corollary~\ref{cor:vspace_qitoprod}.

	Note the $f_{v_i}$ are uniform quasi-isometries, the $\pi_{v_i}$ and $q_{v_i}$ are uniform coarse-Lipschitz and for each $e_i\in \lk(v_i)$, $f_{v_i}(X_{e_i})$ has uniform Hausdorff distance from $\E\times \ell_{e_i}$. Since the inclusions $X_{v_i}\to X$, are uniform quasi-isometric embeddings by Corollary~\ref{cor:ve_spaces_qi_embedded}, we can equip  $X_{v_i}\subseteq X$  with the subspace metric without affecting any of the properties stated above; for instance, the $f_{v_i}$ are still quasi-isometries with uniform constants. Passing to ultralimits and applying Lemma~\ref{lem:induced_bilip}, we obtain the commutative diagram
	\[\begin{tikzcd}
			{\lim_\omega X_{v_i}} & {\mathbb E\times Y_\omega} \\
			{Y_\omega} & {Y_\omega}
			\arrow["f_\omega", from=1-1, to=1-2]
			\arrow["\pi_\omega"', from=1-1, to=2-1]
			\arrow["\text{Id}"', from=2-1, to=2-2]
			\arrow["q_\omega", from=1-2, to=2-2]
		\end{tikzcd}\] where $f_\omega$ is the required bi-Lipschitz equivalence, $\pi_\omega$ is Lipschitz, and $q_\omega$ is the projection.

	By construction, each  $Y_{v_i}$ is the Cayley graph of one of the finitely many non-elementary hyperbolic groups, thus giving (\ref{prop:ultralimit_vertex_space1}).
	Properties (\ref{prop:ultralimit_vertex_space2}) and (\ref{prop:ultralimit_vertex_space3}) follow from the commutativity of the above diagram,
	Theorem~\ref{thm:drutusapir_treegraded} and Corollary~\ref{cor:cay_relhyp}.
\end{proof}

\begin{defn}
	Given $\lim_\omega X_{v_i}\in \cV_\omega$ and $\pi_\omega:\lim_\omega X_{v_i}\to Y_\omega$ as in Proposition~\ref{prop:ultralimit_vertex_space}, we  define an \emph{$\omega$-fiber of $\lim_\omega X_{v_i}$} to be  $\pi_\omega^{-1}(y_\omega)$ for some $y_\omega\in Y_\omega$.
\end{defn}

We prove the following fundamental properties regarding these subsets of $X_\omega$.
\begin{prop}\label{prop:omega-spaces}
	Assume  $\lim_\omega X_{v_i},\lim_\omega X_{v'_i}\in \cV_\omega$  and $\lim_\omega X_{e_i},\lim_\omega X_{e'_i}\in \cE_\omega$.
	\begin{enumerate}
		\item  No element of $\cV_\omega$ is contained in an element of $\cE_\omega$.\label{item:omega-spaces1}
		\item\label{item:omega-spaces2} $\lim_\omega X_{v_i} =\lim_\omega X_{v'_i}$ if and only if $v_i=v'_i$ $\omega$-almost surely.
		\item\label{item:omega-spaces5} If $e_i$ and $e_i'$ are distinct elements of $\lk(v_i)$ $\omega$-almost surely, then $\lim_\omega X_{e_i}\cap \lim_\omega X_{e'_i}$ is either empty or an $\omega$-fiber of $\lim_\omega X_{v_i}$.
		\item $\lim_\omega X_{e_i}=\lim_\omega X_{e'_i}$ if and only if $e_i=e'_i$ or  $e_i=\bar e'_i$ $\omega$-almost surely.\label{item:omega-spaces3}

		\item $\lim_\omega X_{e_i}\subseteq \lim_\omega X_{v_i}$ if and only if  $v_i\in \{({e_i})_-,(e_i)_+\}$ $\omega$-almost surely.\label{item:omega-spaces4} In particular, each element of  $\cE_\omega$ is contained in precisely two elements of $\cV_\omega$.
	\end{enumerate}

\end{prop}
To prove Proposition~\ref{prop:omega-spaces}, we require the following lemma.
\begin{lem}\label{lem:between_vespaces}
	Suppose we have sequences $(a_i), (b_i), (c_i)$ in $VT\sqcup ET$ such that $b_i$ is between $a_i$ and $c_i$ $\omega$-almost surely. Then $\lim_\omega X_{a_i}\cap \lim_\omega X_{c_i} \subseteq \lim_\omega X_{b_i}$, provided these ultralimits exist.
\end{lem}
\begin{proof}
	Let $x_\omega\in\lim_\omega X_{a_i}\cap \lim_\omega X_{c_i}$. Then $x_\omega=[(x_i)]=[(y_i)]$, and for each $i$,  $x_i\in X_{a_i}$ and $y_i\in X_{c_i}$ with $\lim_\omega d(x_i,y_i)/\lambda_i=0$. Since $X_{b_i}$ is between $X_{a_i}$ and $X_{c_i}$ $\omega$-almost surely, we can choose  a sequence $(z_i)$ such that $z_i\in X_{b_i}$ and  $d(x_i,z_i)\leq d(x_i,y_i)$ $\omega$-almost surely. Hence, $[(z_i)]=[(x_i)]=x_\omega\in \lim_\omega X_{b_i}$.
\end{proof}

\begin{proof}[Proof of Proposition~\ref{prop:omega-spaces}]
	(\ref{item:omega-spaces1}): Suppose for contradiction $\lim_\omega X_{v_i}\subseteq \lim_\omega X_{e_i}$. For each $i$, choose $e'_i\in \lk(v_i)$ between $v_i$ and $e_i$.  By Lemma~\ref{lem:between_vespaces}, $\lim_\omega X_{v_i}=\lim_\omega X_{v_i}\cap \lim_\omega X_{e_i}\subseteq \lim_\omega X_{e'_i}$. Since $e'_i\in \lk(v_i)$ for each $i$, $X_{e'_i}\subseteq X_{v_i}$, hence $\lim_\omega X_{v_i}=\lim_\omega X_{e'_i}$. However, by Proposition~\ref{prop:ultralimit_vertex_space},  there is a bi-Lipschitz equivalence $f_\omega:\lim_\omega X_{v_i}\to \E\times Y_\omega$ that takes  $\lim_\omega X_{e'_i}$ to a proper subset  $\E\times \ell\subseteq \E\times Y_\omega$. This contradicts the fact $\lim_\omega X_{v_i}= \lim_\omega X_{e'_i}$.

	(\ref{item:omega-spaces2}): If $v_i\neq v'_i$ $\omega$-almost surely, we choose a sequence of edges $e_i$ such that $e_i$ is between $v_i$ and $v'_i$ $\omega$-almost surely. Lemma~\ref{lem:between_vespaces} implies $\lim_\omega X_{v_i}\cap \lim_\omega X_{v'_i} \subseteq \lim_\omega X_{e_i}$. Thus $\lim_\omega X_{v_i}\neq \lim_\omega X_{v'_i}$, contradicting (\ref{item:omega-spaces1}). The converse is trivial.

	(\ref{item:omega-spaces5}): Let $f_\omega:\lim_\omega X_{v_i}\to \E\times Y_\omega$, $\pi_\omega:\lim_\omega X_{v_i}\to Y_\omega$ and $\cL_\omega$ be as in Proposition~\ref{prop:ultralimit_vertex_space}. By Corollary~\ref{cor:cay_relhyp} and the fact that are only finitely many  $Y_{v_i}$ as we vary $i$, each $Y_{v_i}$ is $\omega$-almost surely the Cayley graph $Y$ of a fixed relatively hyperbolic group and $\ell_{e_i}$ and $\ell_{e'_i}$ are distinct cosets of peripheral subgroups. By Lemma~\ref{lem:ultralimitsvscosets}, $\lim_\omega\ell_{e_i}\neq \lim_\omega\ell_{e'_i}$. Since $\lim_\omega\ell_{e_i}$ and $\lim_\omega\ell_{e'_i}$ are distinct pieces of the tree-graded space $Y_\omega$, their intersection is either empty or a singleton $\{y_\omega\}$. Therefore, $f_\omega(\lim_\omega X_{e_i})\cap f_\omega(\lim_\omega X_{e'_i})$ is either empty or $\E\times \{y_\omega\}$, whence $\lim_\omega X_{e_i}\cap \lim_\omega X_{e'_i}$ is either empty or a fiber of $\lim_\omega X_{v_i}$.

	(\ref{item:omega-spaces3}): Assume that  $e'_i\notin \{e_i,\bar e_i\}$ $\omega$-almost surely. Interchanging $e_i$ and $\bar e_i$ if necessary, which doesn't alter $\lim_\omega X_{e_i}$, we assume $e_i$ is oriented away from $e'_i$ $\omega$-almost surely. Set $v_i=(e_i)_-$. For each $i$, choose $f_i\in \lk(v_i)\setminus\{e_i\}$ between $e_i$ and $e'_i$   $\omega$-almost surely.  Thus by Lemma~\ref{lem:between_vespaces}, $\lim_\omega X_{e_i}\cap \lim_\omega X_{e'_i} \subseteq \lim_\omega X_{f_i}$. Thus $e_i$ and $f_i$ are distinct elements of $\lk(v_i)$  $\omega$-almost surely. Therefore, (\ref{item:omega-spaces5}) implies $\lim_\omega X_{e_i}\cap \lim_\omega X_{f_i}$ is either empty or a fiber of $\lim_\omega X_{v_i}$. In either case, we see  $\lim_\omega X_{e_i}\nsubseteq \lim_\omega X_{f_i}$. Since $\lim_\omega X_{e_i}\cap \lim_\omega X_{e'_i} \subseteq \lim_\omega X_{f_i}$, we have $\lim_\omega X_{e_i}\neq \lim_\omega X_{e'_i}$.  The converse is trivial.

	(\ref{item:omega-spaces4}): Assume $\lim_\omega X_{e_i}\subseteq \lim_\omega X_{v_i}$. Without  altering $\lim_\omega X_{e_i}$, we replace $e_i$ by $\bar e_i$ if needed so that $e_i$ is oriented away from $v_i$. Suppose for contradiction $v_i\neq (e_i)_-$ $\omega$-almost surely. Then there exists a sequence $(f_i)$ of edges  $f_i\in \lk(v_i)\backslash \{e_i\}$  strictly between $e_i$ and $v_i$ $\omega$-almost surely. Therefore, Lemma~\ref{lem:between_vespaces} implies  $\lim_\omega X_{e_i}=\lim_\omega X_{e_i}\cap \lim_\omega X_{v_i}\subseteq \lim_\omega X_{f_i}$. This contradicts 	(\ref{item:omega-spaces5}) and (\ref{item:omega-spaces3}).
\end{proof}
Proposition~\ref{prop:omega-spaces} ensures that the following are well-defined.
\begin{defn}
	Given $\lim_\omega X_{a_i},\lim_\omega X_{b_i},\lim_\omega X_{c_i}\in \cV_\omega\sqcup \cE_\omega$, we say \emph{$\lim_\omega X_{b_i}$ is (strictly) between $\lim_\omega X_{a_i}$ and $\lim_\omega X_{c_i}$} if $b_i$ is (strictly) between $a_i$ and $c_i$ $\omega$-almost surely.
	We say $\omega$-vertex spaces $\lim_\omega X_{v_i}$ and $\lim_\omega X_{w_i}$ are \emph{adjacent} if $v_i$ and $w_i$ are adjacent $\omega$-almost surely.
\end{defn}

Taking ultralimits  of the inequality in  Lemma~\ref{lem:dist_projs}, combined with  Proposition~\ref{prop:ultralimit_vertex_space} and~\ref{prop:omega-spaces}, we deduce:
\begin{cor}\label{cor:dist_projs_cones}
	Then there exists a constant $K\geq 1$ such that the following holds. Let $E_\omega$ be an $\omega$-edge space and let $V^+_\omega,V^-_\omega$ be the distinct $\omega$-vertex spaces containing $E_\omega$.
	Let $\pi_\omega^\pm:V^\pm_\omega\to Y_\omega^\pm$ be the maps in Proposition~\ref{prop:ultralimit_vertex_space}.
	Then for all $x_\omega,y_\omega\in E_\omega$, \[\frac{1}{K}d_\omega(x_\omega,y_\omega) \leq d_{Y^+_\omega}(\pi^+_\omega(x_\omega),\pi^+_\omega(y_\omega))+d_{Y^-_\omega}(\pi^-_\omega(x_\omega),\pi^-_\omega(y_\omega))\leq Kd_\omega(x_\omega,y_\omega).\]
\end{cor}

We also have the following:
\begin{prop}\label{prop:fiber_intersection}
	There is a $K$ such the following holds.  Let $E_\omega$ be an $\omega$-edge space and let $V^+_\omega,V^-_\omega$ be the distinct $\omega$-vertex spaces containing $E_\omega$.
	Let $F^+_\omega$ and $F^-_\omega$ be $\omega$-fibers of $V^+_\omega$ and $V^-_\omega$  respectively. Then
	\begin{enumerate}
		\item\label{prop:fiber_intersection1} $|F^+_\omega\cap F^-_\omega|\leq 1$, with equality if and only if $F^+_\omega$ and $F^-_\omega$ are contained in $E_\omega$.
		\item\label{prop:fiber_intersection2} If $F^+_\omega$ and $F^-_\omega$ are contained in $E_\omega$, there is a $K$-bi-Lipschitz equivalence $g_\omega:E_\omega\to \E^2$ taking $F^+_\omega$ and $F^-_\omega$ to perpendicular geodesics.
		\item\label{prop:fiber_intersection3} $F^+_\omega$ and $F^-_\omega$ are at infinite Hausdorff distance.
	\end{enumerate}
\end{prop}

\begin{proof}
	(\ref{prop:fiber_intersection1}):
	We first prove that $|F^+_\omega\cap F^-_\omega|\leq 1$.
	By Lemma~\ref{lem:between_vespaces}, $V^+_\omega\cap V^-_\omega=E_\omega$. By Proposition~\ref{prop:ultralimit_vertex_space}, $E_\omega$ is a union of $\omega$-fibers of $V^+_\omega$, hence either $F^+_\omega\subseteq E_\omega$ or $F^+_\omega\cap E_\omega=\emptyset$. In the latter case, $ F^+_\omega\cap F^-_\omega\subseteq  F^+_\omega\cap V^-_\omega\subseteq F^+_\omega\cap E_\omega =\emptyset$ and we are done. We thus assume $F^+_\omega\subseteq E_\omega$. Similarly, we assume $F^-_\omega\subseteq E_\omega$.
	Let $x_\omega,y_\omega\in F^+_\omega\cap F^-_\omega$. Let $\pi^+_\omega$ and $\pi^-_\omega$ be as in Corollary~\ref{cor:dist_projs_cones}. As $x_\omega$ and $y_\omega$ are in the same fibers of $V_\omega^{+}$ and of $V_\omega^{-}$, we deduce $\pi^+_\omega(x_\omega)=\pi^+_\omega(y_\omega)$ and $\pi^-_\omega(x_\omega)=\pi^-_\omega(y_\omega)$. Therefore, Corollary~\ref{cor:dist_projs_cones} ensures $x_\omega=y_\omega$.

	(\ref{prop:fiber_intersection2}):	Suppose $E_\omega=\lim_\omega X_{e_i}$,  $V^+_\omega=\lim_\omega X_{v_i}$ and $V^-_\omega=\lim_\omega X_{w_i}$.  Let $a_i$ and $b_i$ be generators of the infinite cyclic subgroups $Z_{v_i}$ and $Z_{w_i}$ respectively. There are sequences of cosets $h_iZ_{v_i}$ and $k_iZ_{w_i}$ such that $F_\omega^+=\lim_\omega h_iZ_{v_i}$ and $F_\omega^-=\lim_\omega k_iZ_{w_i}$.

	As in the proof of Lemma~\ref{lem:dist_projs}, there is a number $N$ independent of $i$ such that $a_i^N$ and $b_i^N$ generate a finite index subgroup $H_i$ of $G_{e_i}$ isomorphic to $\Z^2$.  Equipping $H_i$ with the word metric with respect to $\{a_i^N,b_i^N\}$  and choosing suitable basepoints, we obtain a  $K$-bi-Lipschitz equivalence
	\[
		g_\omega:\lim_\omega \frac{1}{\lambda_i}H_i\to E_\omega
	\]
	with respect to suitably chosen base-points.

	Note $\lim_\omega \frac{1}{\lambda_i}H_i$ is isometric to $\E^2$ equipped with the $\ell_1$-metric. The result follows by observing that $\lim_\omega \frac{1}{\lambda_i}(h_iZ_{v_i}\cap H_i)$ and $\lim_\omega \frac{1}{\lambda_i}(k_iZ_{w_i}\cap H_i)$ are perpendicular geodesics in $\lim_\omega \frac{1}{\lambda_i}H_i$ mapping to $F_\omega^+$ and $F_\omega^-$ respectively.

	(\ref{prop:fiber_intersection3}):	 The claim that $F_\omega^+$ and $F_\omega^-$ are at infinite Hausdorff distance is immediate if they are contained in $E_\omega$. If not, then $F_\omega^+$ and $F_\omega^-$ are both finite Hausdorff distance from fibers $F_\omega'^+$ and $F_\omega'^-$ that are contained in $E_\omega$, so we are done.
\end{proof}

\begin{defn}
	A sequence   $\lim_{\omega}X_{v_{1,i}},\lim_{\omega}X_{v_{2,i}},\dots, \lim_{\omega}X_{v_{n,i}}$ of  $\omega$-vertex spaces  is \emph{consecutive} if the vertices $v_{1,i},v_{2,i},\dots, v_{n,i}$ form a geodesic vertex path in  $T$ $\omega$-almost surely.
\end{defn}

Combining Proposition~\ref{prop:omega-spaces}, Lemma~\ref{lem:between_vespaces} and Proposition~\ref{prop:fiber_intersection}, we deduce:
\begin{cor}\label{cor:consecutive}
	The intersection of four consecutive $\omega$-vertex spaces  has cardinality at most one.
\end{cor}

We now describe separation properties of the $\omega$-edge spaces of $X_\omega$.
Given $E_\omega\in \cE_\omega$ and a choice of  $V^+_\omega\in\cV_\omega$ containing $E_\omega$, we define the \emph{signed distance function} $r:X_\omega\to \R$ based at  $E_\omega$ as follows. We choose sequences $(e_i)$ and $(v_i)$ such that  $E_\omega=\lim_{\omega} X_{e_i}$,  $V^+_\omega=\lim_{\omega} X_{v_i}$ and $(e_i)_+=v_i$. We define $\epsilon_{e_i}(x)=1$ if there exists a  continuous path in $X$ from $x$ to $X_{\bar e_i}$ disjoint from $X_{e_i}$, and $\epsilon_{e_i}(x)=-1$ otherwise. Note that $\epsilon_{e_i}$ is constant on each vertex space of $X$. We define the signed distance function $r$ to be  \[
	r(x_\omega)\coloneqq \lim_\omega \frac{\epsilon_{e_i} (x_i)d(x_i,X_{e_i})}{\lambda_i}  \] for each  $x_\omega=[(x_i)]\in X_\omega$.
It is straightforward to verify that $r$ is well-defined, continuous, and satisfies the property $r^{-1}(0)=E_\omega$.

Suppose $V^-_\omega\in \cV_\omega\backslash \{V^+_\omega\}$ is the other $\omega$-vertex space containing $E_\omega$. Then $r(V^+_\omega\backslash E_\omega)\subseteq (0,\infty)$ and $r(V^-_\omega\backslash E_\omega)\subseteq (-\infty,0)$. Furthermore, the signed distance function obtained by replacing   $V^+_\omega$ with $V^-_\omega$ is $-r$. Thus $r$ depends only on $E_\omega$ up to sign.

\begin{defn}
	Let $E_\omega\in \cE_\omega$ and $r$ be a signed distance function based at $E_\omega$. The \emph{sides} of $E_\omega$ are the sets $r^{-1}(0,\infty)$ and $r^{-1}(-\infty,0)$.
\end{defn}
Since the signed distance function $r$ depends only on $E_\omega$ up to sign, the sides of $E_\omega$ are well-defined.
\begin{defn}
	Let $E_\omega\in \cE_\omega$.
	\begin{enumerate}
		\item We say $A\subseteq X_\omega$ is \emph{essentially split by $E_\omega$} if $A$ intersects both sides of $E_\omega$ non-trivially.

		\item 	We say $A,B$ are  \emph{essentially separated} by $E_\omega$ if $A\backslash E_\omega$ and $B\backslash E_\omega$ are non-empty and lie in distinct sides of $E_\omega$.
	\end{enumerate}
\end{defn}

\begin{lem}\label{lem:esssep_vespace}
	Let $A_\omega, B_\omega\in \cV_\omega\sqcup \cE_\omega$ and $E_\omega=\cE_\omega$ be distinct.
	Then:
	\begin{enumerate}
		\item $A_\omega$ is not essentially split by $E_\omega$
		\item $A_\omega$ and $B_\omega$ are essentially separated by $E_\omega$ if and only if $E_\omega$ is strictly between $A_\omega$ and $B_\omega$.
	\end{enumerate}
\end{lem}
\begin{proof}
	Suppose  $A_\omega=\lim_{\omega} X_{a_i}$,  $B_\omega=\lim_{\omega} X_{b_i}$ and  $E_\omega=\lim_{\omega} X_{e_i}$.
	Let $\epsilon_{e_i}(x_i)$ be as in the definition of the signed distance function $r$ at $E_\omega$. Since each $\epsilon_{e_i}(x_i)$ is constant on each $X_{a_i}$, the sign  of $\epsilon_{e_i}  (x_i)d(x_i,X_{e_i})$ is $\omega$-surely either non-positive or non-negative, hence $r(A_\omega)$ is contained in either $[0,\infty)$ or $(\infty,0]$, thus not essentially split by $E_\omega$. Similarly, $B_\omega$ is not essentially split by $E_\omega$.

	Now suppose $e_i$ is strictly between $a_i$ and $b_i$ $\omega$-almost surely. Then $\epsilon_{e_i}$ has opposite signs on $X_{a_i}$ and $X_{b_i}$ $\omega$-almost surely, whence $A_\omega\backslash E_\omega$ and $B_\omega\backslash E_\omega$ are contained in different sides of $E_\omega$.
	Conversely, if $e_i$ is not between  $a_i$ and $b_i$ $\omega$-almost surely, then $\epsilon_{e_i}$ has the same sign on $X_{a_i}$ and $X_{b_i}$ $\omega$-almost surely. Thus $A_\omega\backslash E_\omega$ and $B_\omega\backslash E_\omega$ are contained in the same side of $E_\omega$.
\end{proof}

\begin{defn}
	Let $V_\omega\in \cV_\omega$ and $\pi_\omega:V_\omega\to Y_\omega$ be as in Proposition~\ref{prop:ultralimit_vertex_space}. Suppose  $E^-_\omega,E^+_\omega\in \cE_\omega$ are distinct and contained in $V_\omega$, and let $y_\omega$ be the point of $\pi_\omega(E^-_\omega)$ closest to $\pi_\omega(E^+_\omega)$, which is unique by Proposition~\ref{prop:ultralimit_vertex_space}. We say $F_\omega\coloneqq \pi_\omega^{-1}(y_\omega)$ is \emph{the $V_\omega$-fiber of $E_\omega^-$ closest to $E_\omega^+$}.
\end{defn}

\begin{lem}\label{lem:closest_fiber}
	Let $V_\omega\in \cV_\omega$ and suppose $E^-_\omega,E^+_\omega\in \cE_\omega$ are distinct and contained in $V_\omega$. Let $F_\omega$ be the $V_\omega$-fiber of $E_\omega^-$ closest to $E_\omega^+$. Then any continuous path in $X_\omega$ from $E_\omega^-$ to $E_\omega^+$ intersects $F_\omega$.
\end{lem}
\begin{proof}
	For each $x\in V_\omega\setminus F_\omega$, let $C_x$ denote the path component of $V_\omega\setminus F_\omega$ containing $x$. We endow $V_\omega\setminus F_\omega$ with an equivalence relation $\sim$ defined as follows. If $x,y\in V_\omega\setminus F_\omega$, we say $x\sim y$ if either:
	\begin{enumerate}
		\item $C_x=C_y$;
		\item there is some $E_\omega\in \cE_\omega$  contained in $V_\omega$ and intersecting both  $C_x$ and $C_y$.
	\end{enumerate}
	The tree-graded configuration described in  Proposition~\ref{prop:ultralimit_vertex_space} ensures that this is a well-defined equivalence relation and that every equivalence class consists of either a single path component of $V_\omega\setminus F_\omega$ or a union of two such components. Let $\cC$ be the set of equivalence classes.

	Let $\gamma:[0,L]\to X_\omega$ be a path from $E_\omega^-$ to $E_\omega^+$.  As $E_\omega^-$ and $E_\omega^+$ are distinct, we have that $\gamma(0),\gamma(L)$ are contained in distinct elements $C^-$ and $C^+$ of $\cC$.
	Set
	\[
		r\coloneqq\sup\{t\in [0,L]\mid \gamma(t)\in C^-\}
	\]
	Note that $\gamma(r)\in \bar{C^-}=C^-\cup F_\omega$. If $\gamma(r)\in F_\omega$ we are done, so we may assume that $\gamma(r)\in C^-$.  If $r$ is a limit point of $(r,L]\cap \gamma^{-1}(V_\omega)$, then  $\gamma(r)\in \bar {C^-}\cap \bar{V_\omega\backslash C^-}=F_\omega$, contradicting the previous assumption. We can thus choose  $s\in (r,L]$ such that $\gamma(r),\gamma(s)\in V_\omega$ and $\gamma((r,s))$ is disjoint from $V_\omega$. We suppose also $\gamma(s)$ is not in $F_\omega$, otherwise, we are done.

	Pick $t\in (r,s)$ and choose  sequences $(v_i)$ and $(w_i)$ in $VT$ such that $V_\omega=\lim_\omega X_{v_i}$ and $\gamma(t)\in \lim_\omega X_{w_i}\in \cV_\omega$.  Since $\gamma(t)\notin V_\omega$, $v_i\neq w_i$ $\omega$-almost surely.  Pick $f_i\in \lk(v_i)$ strictly between $v_i$ and $w_i$ $\omega$-almost surely. By Lemma~\ref{lem:esssep_vespace}, $V_\omega$ and $\lim_\omega X_{w_i}$ are essentially separated by  $\lim_\omega X_{f_i}\subseteq \lim_{\omega} X_{v_i}$. Since $\gamma_{[r,t]}$ and $\gamma_{[t,s]}$ are paths between $V_\omega$ and $\lim_\omega X_{w_i}$, intersecting $V_\omega$ only at $\gamma(r)$ and $\gamma(s)$ respectively, we see $\gamma(r),\gamma(s)\in \lim_\omega X_{f_i}$.  The definition of $\cC$ ensures that $\gamma(r)$ and $\gamma(s)$ are both in $C^-$, contradicting the choice of $r$.%
\end{proof}

\begin{defn}
	A continuous path $p \coloneqq I\to X_\omega$ \emph{has no essential backtracking} if for every $A\in  \cE_\omega\sqcup \cV_\omega$, the preimage $p^{-1}(A)$ is an interval.
\end{defn}

\begin{lem}Every $x_\omega,y_\omega\in X_\omega$ can be joined by a Lipschitz path $\gamma$ with no essential backtracking.\label{lem:good_ultracone}
\end{lem}
\begin{proof}
	Suppose $x_\omega\in \lim_\omega X_{v_i}$ and $y_\omega\in \lim_\omega X_{w_i}$.
	We set \[K\coloneqq \{\lim_\omega X_{a_i}\in \cV_\omega\mid\text{$a_i\in [v_i,w_i]$ $\omega$-almost surely}\}\]  and define a total order $\leq$ on $K$ as follows:
	\[
		\text{$\lim_\omega X_{a_i}\leq \lim_\omega X_{b_i}$ if $a_i\in [v_i,b_i]$ $\omega$-almost surely}.
	\] If $\lim_\omega X_{a_i}\leq \lim_\omega X_{b_i}$, then   $b_i\in[a_i, w_i]$ $\omega$-almost surely.
	For $k,k'\in K$, we write $k<k'$ if $k\leq k'$ and $k\neq k'$. The total  order  has a minimal element $m\coloneqq\lim_\omega X_{v_i}$ and a maximal element $M\coloneqq\lim_\omega X_{w_i}$.

	If $k=\lim_\omega X_{a_i}\in K\backslash \lim_\omega X_{w_i}$, we define its \emph{successor} $S(k)\coloneqq\lim_\omega X_{b_i}$, where $b_i$ is first vertex other than $a_i$ in the geodesic  $[a_i,w_i]$ $\omega$-almost surely.
	Thus, $k<S(k)$ and there is no other $k'\in K$ with $k<k'<S(k)$.

	The function $S:K\backslash\{M\}\to K\backslash\{m\}$ is a bijection with  inverse $S^{-1}$ defined on $K\backslash \{m\}$.  Furthermore, if $e_i$ is the edge with endpoints $a_i$ and $b_i$, we see that $E_k\coloneqq \lim_\omega X_{e_i}$ is the unique element of $\cE_\omega$ essentially separating $k$ and $S(k)$. By Proposition~\ref{prop:omega-spaces}, $E_k$ is contained in both $k$ and $S(k)$.

	For each $k\in K$, let $\pi_{k}:k\to Y_k$ be as in Proposition~\ref{prop:ultralimit_vertex_space}.
	\begin{itemize}
		\item
		      For $k\neq m,M$, let $F_k^-$ be the $\omega$-fiber of $k$ contained in  $E_{S^{-1}(k)}$ and  closest to $E_{k}$, and let $F_k^+$ be the $\omega$-fiber of $k$ contained in  $E_{k}$ and closest to $E_{S^{-1}(k)}$.  We note that $F_{k}^-=F_k^+$ if and only if $E_k\cap E_{S^{-1}(k)}\neq \emptyset$.
		\item
		      For $k=m$, let $F_k^+$ be the $\omega$-fiber of $k$ contained in $E_k$ and closest to $x_\omega$.

		\item Similarly, for $k=M$, let $F_k^-$ be the $\omega$-fiber of $k$ contained in $E_{S^{-1}(k)}$ and closest to $y_\omega$.
	\end{itemize}
	For each $k\neq M$, let $x_k\in E_k$ be the unique element of $F_k^+\cap F_{S(k)}^-$, which exists by Proposition~\ref{prop:fiber_intersection}.

	Let $\gamma:[0,L]\to X_\omega$ be a geodesic in $X_\omega$ from $x_\omega$ to $y_\omega$. For each $k\in K$, set
	\[t_k\coloneqq \sup\{t\in [0,L]\mid \gamma(t)\in k\}\]
	Since $k$ is closed, $\gamma(t_k)\in k$.

	Let $k\in K\backslash \{M\}$. As $\gamma((t_{k},L])$ is disjoint from $E_k\subseteq k$, it is contained in the side of $E_k$ containing $y_\omega$. Since $E_k$ essentially separates $k$ from $y_\omega$ by Lemma~\ref{lem:esssep_vespace}, we must have $\gamma(t_k)\in E_k$. If $k'\leq k$, then as $E_k$ essentially separates $k'$ from $y_\omega$  by Lemma~\ref{lem:esssep_vespace}, we must have $t_{k'}\leq t_k$.  By Lemma~\ref{lem:closest_fiber},  any path from $E_k$ to $E_{S(k)}$ must intersect $F_{S(k)}^-$. As $\gamma|_{(t_k,t_{S(k)}]}$ does not intersect $E_k$,  it follows that $\gamma(t_k)\in F_{S(k)}^-$.
	Moreover, Lemma~\ref{lem:closest_fiber} and the fact $F_{k}^+$ is closed   ensures there is a minimal $r_{S(k)}\in [t_k,t_{S(k)}]$ with $\gamma(r_{S(k)}) \in F_{S(k)}^+$. We can  also choose a minimal $r_m\in [0,t_m]$ with $\gamma(r_{m}) \in F_{m}^+$.

	\begin{figure}[htb]
		\centering
		\resizebox{0.5\textwidth}{!}{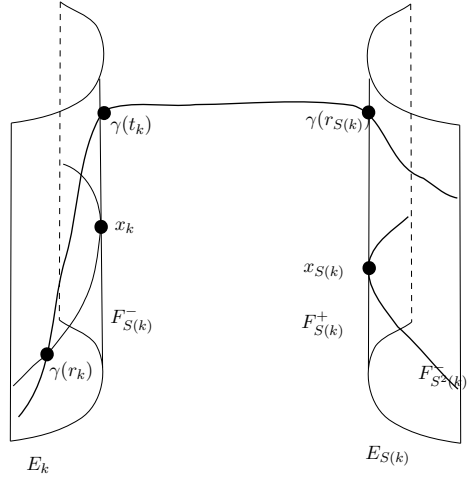}
		\caption{The picture illustrates  two fibers $F_{S(k)}^{-}$ and $F_{k}^{+}$ of $E_k$ (resp.\, $F^{+}_{S(k)}$ and $F^{-}_{S^2(k)}$ of $E_{S(k)}$) intersecting at $x_k$ (resp.\, $x_{S(k)}$). Here $F_{S(k)}^{-}$ is the $S(k)$-fiber of $E_k$ closest to $E_{S(k)}$ and $F_{S(k)}^{+}$ is the $S(k)$-fiber of $E_{S(k)}$ closest to $E_{k}$. The path $\gamma$ intersects $F_{S(k)}^{-}$ and leaves $F_{S(k)}^{-}$ at $\gamma(t_k)$.}\label{figure:thick}
	\end{figure}

	Pick a constant $L_1$ such that Corollary~\ref{cor:dist_projs_cones} holds (with $L_1$ in place of $K$) and each $\pi_k$ is $L_1$-Lipschitz. Then for each $k\in K\backslash \{M\}$, we have
	$\pi_k(\gamma(r_k))=\pi_k(x_k)$ and $\pi_{S(k)}(\gamma(t_k))=\pi_{S(k)}(x_k)$, hence
	\begin{align*}
		d_{X_\omega}(\gamma(r_k),x_{k}) & \leq L_1d_{Y_{S(k)}}(\pi_{S(k)}(x_k),\pi_{S(k)}(\gamma(r_k)))        \\
		                                & =L_1d_{Y_{S(k)}}(\pi_{S(k)}(\gamma(t_k)),\pi_{S(k)}(\gamma(r_k)))    \\
		                                & \leq L_1^2 d_{X_\omega}(\gamma(t_k),\gamma(r_k))\leq L_1^2(t_k-r_k).
	\end{align*}
	An identical argument gives $d_{X_\omega}(\gamma(t_k),x_{k})= L_1^2(t_k-r_k)$.

	It follows from Propositions~\ref{prop:ultralimit_vertex_space} and~\ref{prop:fiber_intersection}  that there is a constant $L_2\geq 1$, depending only on $X_\omega$, such that the following hold:
	\begin{enumerate}
		\item  For each $k\in K$, there is an $L_2$-bi-Lipschitz geodesic
		      \[
			      \gamma_k:[r_k,t_k]\to X_\omega
		      \]
		      given by concatenating the bi-Lipschitz geodesic segments  $[\gamma(r_k),x_k]_{F_{k}^+}\cdot [x_k,\gamma(t_k)]_{F_{S(k)}^-}$ and scaling the domain.
		\item For each $k\in K\backslash\{M\}$, there is an $L_2$-bi-Lipschitz geodesic
		      \[
			      \delta_k:[t_k,r_{S(k)}]\to X_\omega
		      \]
		      from $\gamma(t_k)$ to $\gamma(r_{S(k)})$ contained in $S(k)$, such that $\pi_{S(k)}\circ \delta_k$ is an arc in $Y_{S(k)}$ from $\pi_{S(k)}(F_{S(k)}^-)$ to $\pi_{S(k)}(F_{S(k)}^+)$.
	\end{enumerate}

	For $k=m$, we also pick an $L_2$-bi-Lipschitz geodesic $\gamma_0:[0,r_m]\to X_\omega$, contained in $m$ and intersecting each fiber of $m$ at most once.

	We can thus define a new path $\gamma':[0,L]\to X_\omega$ such that $\gamma'|_{[r_k,t_k]}=\gamma_k$, $\gamma'|_{[t_k,r_{S(k)}]}=\delta_k$ and $\gamma'|_{[0,t_m]}=\gamma_0$. Then $\gamma'$ is itself $L_2$-Lipschitz as a concatenation of $L_2$-Lipschitz paths.

	The choice of $\gamma_k$ and $\delta_k$ readily imply that  for all  $A_\omega\in \cV_\omega\sqcup \cE_\omega$,  $\gamma'^{-1}(A_\omega)$ is an interval. For instance, for  each $k\in K$, the choice of the $r_k$ and $t_{S(k)}$ ensure that $\gamma'^{-1}(S(k))=[r_k,t_{S(k)}]$ $\gamma'^{-1}(E_k)=[r_k,t_k]$.
\end{proof}

We use this to show:
\begin{prop}\label{prop:essential_split}
	If $A\subseteq X_\omega$ is not essentially split by any $\omega$-edge space, then $A$ is contained in an  $\omega$-vertex space.
\end{prop}
\begin{proof}
	The proof is similar to that of~\cite[Lemma 3.4]{kapovichleeb1997Quasiisometries}, making use of Lemma~\ref{lem:good_ultracone} to account for the fact that  we no longer have convexity of $\omega$-vertex spaces.

	There is nothing to prove if $A$ consists of a single point, so let
	$x_\omega$ and $y_\omega$ be two distinct points of $A$.

	We claim $\{x_\omega, y_\omega\}$ is contained in an $\omega$-vertex space. Assume for contradiction this is not the case. There exist distinct edge spaces $E_\omega$ and $E'_\omega$ with $x_\omega\in E_\omega$ and $y_\omega\in E'_\omega$.
	By Lemma~\ref{lem:good_ultracone}, there is a path  $\gamma:[0,1]\to X_\omega$ connecting $x_\omega$ to $y_\omega$ with no essential backtracking.  Let $R$ be the set of $V_\omega\in \cV_\omega$ that intersect $\gamma(0,1)$ and are between $E_\omega$ and $E'_\omega$. Since $A$ is not contained in any element of $R$ and is not essentially split by any $\omega$-edge space, every element of $R$ contains at least one of $\{x_\omega, y_\omega\}$.   Since $\gamma$ has no essential backtracking, every element of $R$ contains a subset of the form $\gamma([0,t])$ or $\gamma([t,1])$ for some $0<t<1$. As four consecutive components of $R$ contain at most one common point by Corollary~\ref{cor:consecutive},  $R$ is finite, hence contains adjacent $V_\omega$ and $V'_\omega$ containing $x_\omega$ and $y_\omega$ respectively. Let $E''_\omega$ be the $\omega$-edge space between $V_\omega$ and $V'_\omega$, which essentially separates $V_\omega$ and $V'_\omega$.  Since $\{x_\omega, y_\omega\}$ is not contained in either $V_\omega$ or $V'_\omega$, we deduce $\{x_\omega, y_\omega\}$ must be essentially split by $E''_\omega$. This is the required contradiction.

	Now suppose $\{x_\omega, y_\omega\}$ is contained in $V_\omega$, which we  denote by $V_{0,\omega}$ for notational convenience. We claim that $A$ is contained in an $\omega$-vertex space. We assume for contradiction this is not the case. Then in particular,   $A\nsubseteq V_{0,\omega}$. Thus there is some $E_{1,\omega}$ contained in $V_{0,\omega}$  such that $V_{0,\omega}$ and some $a_\omega\in A$  are essentially separated by $E_{1,\omega}$. Let $V_{1,\omega}\neq V_{0,\omega}$ be the other $\omega$-vertex space containing $E_{1,\omega}$. Since $A$ is not essentially split by $E_{1,\omega}$, $\{x_\omega, y_\omega\}\subseteq V_{1,\omega}$. Since $A\nsubseteq V_{1,\omega}$, then we argue as above to find $V_{2,\omega}$ such that $\{x_\omega, y_\omega\}\subseteq V_{2,\omega}$ and $V_{0,\omega},V_{1,\omega},V_{2,\omega}$ are consecutive. We continue in this way, obtaining four consecutive components containing $\{x_\omega, y_\omega\}$, contradicting Corollary~\ref{cor:consecutive}.
\end{proof}

\begin{prop}\label{prop:component_bdry}
	Let $C$ be a path component of $X_\omega\backslash E_\omega$. Suppose $C$ is contained in the side of $E_\omega$ containing $V_\omega\backslash E_\omega$, where $E_\omega$ is contained in $V_\omega$. Then $\partial C$ is contained  in a fiber of $V_\omega$.
\end{prop}

\begin{proof}
	Let $f_\omega:V_\omega\to \E\times Y_\omega$ be a bi-Lipschitz equivalence as in Proposition~\ref{prop:ultralimit_vertex_space}. By Proposition~\ref{prop:ultralimit_vertex_space},  $f_\omega(E_\omega)=\E\times \ell$ for some bi-infinite geodesic $\ell$ in $T_\omega$. Let $\pi_\omega:V_\omega\to Y_\omega$ be the quotient map and let $\proj_\ell:Y_\omega \to \ell$ be the closest point projection. Set $S_\omega$ to be the side of $E_\omega$ containing $V_\omega\backslash E_\omega$ and set $H_\omega\coloneqq E_\omega\sqcup S_\omega$.
	To show that $\partial C$ is contained in a fiber of $V_\omega$, we are going to define a continuous function \[ \vartheta:H_\omega\to \ell\] and show that this map is constant on $C$.

	Our desired map $\vartheta$ is defined as follows. We define  $ \vartheta|_{V_\omega}$ to be the composition $\proj_\ell\circ \pi_\omega:V_\omega\to \ell$, which is clearly continuous.  For each edge $E'_\omega\neq E_\omega$ incident to $V_\omega$, let $S'_\omega$ be the side of $E'_\omega$ that doesn't contain $V_\omega\backslash E'_\omega$.   By Proposition~\ref{prop:ultralimit_vertex_space},  $f_\omega(E'_\omega)=\E\times \ell'$ for some geodesic $\ell'$ in $Y_\omega$ satisfying $|\ell\cap \ell'|\leq 1$; thus $\proj_\ell(\ell')$ is a single point.  We therefore define $\vartheta(x)$ by  $\vartheta(x)=\proj_\ell(\ell')$ for all $x\in {\bar{S'_\omega}}$. This yields a  well-defined continuous function $\vartheta:H_\omega\to \ell$.

	We will show that  $\vartheta$ is constant on $C$. Indeed, let $x_\omega,y_\omega\in C$ and let $p:[0,1]\to C$ be a path from $x_\omega$ to $y_\omega$. Each component of $p([0,1])\backslash V_\omega$ is contained in some side $S_\omega'$ as above, hence $\vartheta$ is constant on each such component. Let
	\[
		r\coloneqq \sup\{t\mid (\vartheta\circ p)[0,t]=\vartheta(x_\omega)\}
	\]
	We claim $r=1$. Assume for contradiction $r<1$. If $p(r)\notin V_\omega$, then  $\vartheta$ is constant on  $p(r-\epsilon,r+\epsilon)$ for some $\epsilon>0$, contradicting the definition of $r$ and the assumption $r<1$.

	If $p(r)\in V_\omega$, then as $Y_\omega$ is a tree and $\pi_\omega\circ p$ is continuous, we can choose  $\epsilon>0$ small enough such that the image of $p(r-\epsilon,r+\epsilon)\cap V_\omega$ under $\proj_\ell\circ \pi_\omega$ is a point of $\ell$. We conclude that $\vartheta$ is constant on  $p(r-\epsilon,r+\epsilon)\cap V_\omega$.
	As  $\vartheta$ is also constant on components of $p(r-\epsilon,r+\epsilon)\backslash V_\omega$, it follows $\vartheta$ is constant on $p(r-\epsilon,r+\epsilon)$. This also contradicts the definition of $r$. Thus $r=1$, and  so by continuity,  $\vartheta(x_\omega)=\vartheta(y_\omega)$.

	Since $x_\omega,y_\omega\in C$ were arbitrary, $\vartheta(C)=\{z\}$ for some $z\in \ell$. By continuity, we conclude  $\vartheta(\partial C)=\{z\}$. Since  $\vartheta|_{E_\omega}$ coincides with the  projection $\pi_\omega|_{E_\omega}$,  $\partial C$ is contained in the fiber $\pi_\omega^{-1}(z)$.
\end{proof}

We recall the following lemma:
\begin{lem}[{\cite[Lemma 3.3]{kapovichleeb1997Quasiisometries}}]\label{lem:sep_fiber}
	Let $Y_\omega$ be a geodesically complete tree that branches everywhere and let $C\subseteq  \E$ be a closed subset.
	Assume that $g: C \to \E\times Y_\omega$ is a bi-Lipschitz embedding whose image separates.
	Then $C = \E$ and $g(C)$ is a fiber $\E\times \{x\}$.
\end{lem}

The statement of the following proposition is similar to {\cite[Lemma 3.10]{kapovichleeb1997Quasiisometries}}.
\begin{prop}\label{prop:biLip-TxE}
	Let $Y_\omega$ be a geodesically complete tree that branches everywhere. The image of every bi-Lipschitz embedding $f:\E\times Y_\omega\to X_\omega$ is contained in an $\omega$-vertex space.
\end{prop}
\begin{proof}
	By Proposition~\ref{prop:essential_split}, it is sufficient to show $\Image(f)$ is not essentially split by any  $\omega$-edge space. Suppose for contradiction  $E_\omega$ essentially splits $\Image(f)$. Let $S^+_\omega$ and $S^-_\omega$ be the two sides of $E_\omega$, and let $V^\pm_\omega$  be the $\omega$-vertex spaces containing  $E_\omega$ such that $V^\pm_\omega\backslash E_\omega\subseteq S^\pm_\omega$.

	Let $A^\pm$ be a path component of $\Image(f)\backslash E_\omega$ contained in $ S^\pm_\omega$. Proposition~\ref{prop:component_bdry} ensures  $\partial A^\pm $ is contained in a fiber of $V_\omega^\pm$.
	Since $f^{-1}(\partial A^\pm)$ separates $\E\times T$, applying Lemma~\ref{lem:sep_fiber}, we deduce that $\partial A^\pm$ is a fiber of $V_\omega^\pm$ and that $f^{-1}(\partial A^\pm)=\E\times \{x^\pm\}$. Since $\E\times \{x^+\}$ and $\E\times \{x^-\}$ are at finite Hausdorff distance, so are $\partial A^+$ and $\partial A^-$ (as $f$ is a bi-Lipschitz embedding). This cannot be the case, since a fiber of $V_\omega^+$ and a fiber of $V_\omega^-$ cannot be at finite Hausdorff distance by Proposition~\ref{prop:fiber_intersection}.
\end{proof}
Since every $\omega$-vertex space is bi-Lipschitz equivalent to $\E\times Y_\omega$, we use Proposition~\ref{prop:biLip-TxE} applied to a bi-Lipschitz equivalence $f:X_\omega\to X'_\omega$ and its inverse to deduce:
\begin{cor}
	For every bi-Lipschitz equivalence $f:X_\omega\to X'_\omega$ and every $\omega$-vertex space $V_\omega$, there exists an $\omega$-vertex space $V'_\omega$ such that $f(V_\omega)= V'_\omega$.
\end{cor}

Since every $\omega$-edge space is the intersection of its incident $\omega$-vertex spaces, we deduce:
\begin{cor}\label{cor:bilip_preserve_edge}
	For every bi-Lipschitz equivalence $f:X_\omega\to X_\omega$ and every $\omega$-edge space $E_\omega$, there exists an $\omega$-edge space $E'_\omega$ such that $f(E_\omega)= E'_\omega$.
\end{cor}

\subsection{Quasi-isometries preserve edge spaces}
We fix a non-principal ultrafilter $\omega$ and two admissible groups $G$ and $G'$. Let $(X,T)$ and $(X',T')$ be the associated trees of spaces.
In this section, we are going to prove the following proposition.

\begin{prop}\label{prop:edgespaces_preserved_QI}
	For every $K\geq 1$ and $A\geq 0$, there exists a constant $B=B(K,A,X, X')$ such that for every edge space $E$ of $X$ and $(K,A)$-quasi-isometry $f:X\to X'$, there exists an edge space $E'$ of $X'$ such that $d_\Haus(f(E),E')\leq B$.
\end{prop}

The proof of Proposition~\ref{prop:edgespaces_preserved_QI} is similar to the proof of~\cite[Corollary 8.33]{frigeriolafontsistosro2015rigidity}.

\begin{lem}\label{lem:basepoint_choose}
	If Proposition~\ref{prop:edgespaces_preserved_QI} is not true, then there exists an edge space $E$ and a sequence $f_i:X\to X'$ of $(K,A)$ quasi-isometries such that $f_i(E)\nsubseteq N_i(E')$ for any $E'\in \cE'$.
	Moreover, there exists a point $b\in E$, a sequence $(b_i)$ in $E$, and  a sequence $\hat E_i$ of edge spaces of $X'$ such that the following hold for every $i\in \N$:
	\begin{enumerate}
		\item $d(f_i(b_i),\hat E_i)\geq i$.
		\item If $x \in E$ and $d(x,b)\leq d(b_i,b)$, then $d(f_i(x),\hat E_i)\leq i+K+A$.
		\item $\lim_\omega \frac{d(b_i,b)}{i}=\infty$.
	\end{enumerate}
\end{lem}
\begin{proof}
	By Proposition~\ref{prop:wellknownfacts}, we see that Proposition~\ref{prop:edgespaces_preserved_QI} is true if there exists a constant $B$ such that for every $(K,A)$-quasi-isometry $f:X\to X'$ and edge space $E$ of $X$, there is an edge space $E'$ of $X'$ with $f(E)\subseteq N_B(E')$.
	Therefore, if Proposition~\ref{prop:edgespaces_preserved_QI} is not true, then for each $i$, there is a $(K,A)$-quasi-isometry $f_i:X\to X'$ and some edge space  $E_i$ of $E$ such that $f_i(E_i)\nsubseteq N_i(E')$ for any edge space $E'$ of $X'$.
	Since there are only finitely many $G$-orbits of edge spaces in $X$, after passing to a subsequence and precomposing $f_i$ with left multiplication by an element of $G$, we can suppose all the $E_i$ are equal to some $E$.

	Pick $b\in E$. 	Consider the asymptotic cone $X_\omega$ of $X$ with basepoint $(b)$ and scaling sequence $(i)$.
	Let $X'_\omega$ be the asymptotic cone of $X'$ with basepoints $(f_i(b))$ and scaling constants $(i)$.
	Then $(f_i)$ induces  a bi-Lipschitz equivalence $f_\omega:X_\omega\to X'_\omega$. Let $E_\omega=\lim_\omega E\subseteq X_\omega$. By Corollary~\ref{cor:bilip_preserve_edge}, there is an $\omega$-edge space $\hat E_\omega=\lim_\omega \hat E_i$ such that $f_\omega(E_\omega)=\hat E_\omega$.
	Our hypotheses on $f_i$ ensure that for each $i$, $f(E)\nsubseteq N_i(\hat E_i)$. We pick $b_i\in E$ with $d(b_i,b)$ minimal such that
	\[
		d(f_i(b_i),\hat E_i)\geq i.
	\]

	The choice of $i$ ensures that for every $x\in E$ with $d(x,b)<d(b_i,b)$, we have $d(f_i(x),\hat E_i)\leq i$. Now for each $x\in E$ with $d(x,b)\leq d(b_i,b)$, there is some $x'\in E$ with $d(x,x')\leq 1$ and $d(x',b)<d(b_i,b)$. Hence \[
		d(f_i(x),\hat E_i)\leq d(f_i(x),f_i(x'))+d(f_i(x'),\hat E_i)\leq i+K+A
	\]

	Finally, suppose for contradiction  $\lim_\omega \frac{d(b_i,b)}{i} <\infty$. Then $[(b_i)]\in E_\omega$, so that $f_\omega([(b_i)])=[(f_i(b_i))]\in \hat E_\omega$. This leads to a contradiction, since the condition $d(f_i(b_i),\hat E_i)\geq i$ for all $i$ ensures that  \[\lim_\omega \frac{d(f_{i}(b_i),\hat E_i)}{i} \geq 1.\qedhere\]
\end{proof}

The proof of the following lemma is routine, so we leave as an exercise.
\begin{lem}\label{lem:limit_halfspace}
	Let $X=\E^2$. Suppose $(\mu_i)$ is a sequence in $\R_{>0}$ such that $\lim_\omega \frac{\mu_i}{i}=\infty$. Let $b_i=(0,\mu_i)\in \E^2$ and set $A_i=\{y\in X\mid d(y,b_i)\leq \mu_i\}$. If $X_\omega$ is the asymptotic cone of $X$ with respect to the base-point $((0,0))$ and scaling sequence $(i)$, then $X_\omega$ can be canonically be identified with $\E^2$, and $\lim_\omega A_i $ is the upper  half-space $\{(x,y)\in \E^2\mid y\geq 0\}$.
\end{lem}

We are now ready for the proof of Proposition~\ref{prop:edgespaces_preserved_QI}.
\begin{proof}[{Proof of Proposition~\ref{prop:edgespaces_preserved_QI}}]
	Suppose for contradiction Proposition~\ref{prop:edgespaces_preserved_QI} is not true. Pick $E$, $(f_i)$, $(\hat E_i)$,  $b$, $(b_i)$ as in Lemma~\ref{lem:basepoint_choose}.
	We let $X_\omega$ be the asymptotic cone of $X$ with respect to the basepoints $(b_i)$ and scaling sequence $(i)$, and let $X'_\omega$ be the asymptotic cone of $X'$ with respect to the basepoints $(f_i(b_i))$ and scaling sequence $(i)$. The ultralimit of $(f_i)$ induces a bi-Lipschitz equivalence $f_\omega:X_\omega\to X'_\omega$.

	Let $E_\omega=\lim_\omega E\subseteq X_\omega$ and $\hat E_\omega=\lim_\omega \hat E_i\subseteq X'_\omega$. Since $b_i\in E$ and $d(f_i(b_i),\hat E_i)\leq i+K+A$, we see $E_\omega$ and $\hat E_\omega$ are non-empty, hence are $\omega$-edge spaces of $X_\omega$ and $X'_\omega$.

	By Corollary~\ref{cor:bilip_preserve_edge}, there exists an $\omega$-edge space $E'_\omega=\lim_\omega E'_i$ of $X'_\omega$ such that $f_\omega(E_\omega)=E'_\omega$. Since $[(b_i)]\in E_\omega$, we see $f_\omega([(b_i)])=[(f_i(b_i))]\in E'_\omega$. As  $d(f_i(b_i),\hat E_i)\geq i$, we deduce $d_\omega (f_\omega([(b_i)]),\hat E_\omega)\geq 1$. In particular,  $\hat E_\omega\neq E'_\omega$.

	Let $V_\omega$ be the $\omega$-vertex space that is both incident to $\hat E_\omega$ and between $E'_\omega$ and $\hat E_\omega$. By Lemma~\ref{lem:closest_fiber}, there is a fiber $F_\omega$ of $V_\omega$ such that every path from $\hat E_\omega$ to $E'_\omega$ passes through $F_\omega$.
	Let $A_i=\{x\in E\mid d(x,b)\leq d(b_i,b)\}$ and $A_\omega\coloneqq  \lim_\omega A_i$. By Lemma~\ref{lem:limit_halfspace}, there is a bi-Lipschitz equivalence $E_\omega\to \E^2$ that sends to $A_\omega$ to a half-space $\E^2_{\geq 0}$ in $\E^2$.
	We have $f_\omega(A_\omega)\subseteq E'_\omega$. Moreover, by Lemma~\ref{lem:basepoint_choose}, we have $d(f_i(a),\hat E_i)\leq i+K+A$ for all $a\in A_i$. Thus $d'_\omega(f_\omega(A_\omega),\hat E_\omega)\leq 1$, hence $f_\omega(A_\omega)\subseteq E'_\omega\cap N_1(\hat E_\omega)$.
	Since every path from $\hat E_\omega$ to $E'_\omega$ passes through the fiber $F_\omega$, we have $f_\omega(A_\omega)\subseteq N_1(F_\omega)$. Thus there is a quasi-isometric embedding $A_\omega\to F_\omega$. Lemma~\ref{lem:qi_halfspace} says that this is impossible as $A_\omega$ is bi-Lipschitz equivalent to a half-space $\E^2_{\geq 0}$ and $F_\omega$ is bi-Lipschitz equivalent to $\E$.
\end{proof}

\subsection{Quasi-isometries induce automorphisms of the Bass--Serre tree}

\begin{thm}\label{thm:qipreservespieces}
	We fix two admissible groups $G$ and $G'$, and let $(X,T)$ and $(X',T')$ be the associated trees of spaces. For any $(K,A)$-quasi-isometry $f:X\to X'$, there is a constant $B=B(K,A,X, X')$ such that the following holds. There is a tree isomorphism $f_*:T\to T'$ such that
	\[d_\Haus(f(X_x),X'_{f_*(x)})\leq B
	\]
	for every $x\in VT\cup ET$.
\end{thm}

\begin{proof}
	It follows Proposition~\ref{prop:edgespaces_preserved_QI} that there is a map $f_*:ET\to ET'$ and a constant $B=B(K,A,X,X')$ such that $d_\Haus(f(X_e),X_{f_*(e)})\leq B$.
	Lemmas~\ref{lem:edge_vertex_int_admissible} and~\ref{lem:subgp_commensurability} imply that no  two edge spaces are at finite Hausdorff distance, so $f_*$ is well-defined and injective. Applying Proposition~\ref{prop:edgespaces_preserved_QI} to a coarse inverse to $f$, we see that $f_*$ is a bijection from $ET$ to $ET'$.

	We now make use of the notion of coarse intersection of subspaces; see~\cite{moshersageevwhyte2011quasiactions} for a comprehensive treatment of coarse intersection.
	Given a metric space $X$ and two subspaces $A$ and $B$ of $X$, we say $A$ and $B$ have \emph{unbounded coarse intersection} if there exists $r$ such that $N_r(A)\cap N_r(B)$ is unbounded. It is easy to see that if $f:X\to X'$ is a quasi-isometry such that $f(A)$ and $f(B)$ have finite Hausdorff distance from $A'$ and $B'$, then $A$ and $B$ have unbounded coarse intersection if and only if $A'$ and $B'$ do.

	Consider the set $R\coloneqq\{\lk(v)\mid v\in VT\}$, which is a collection of subsets of $ET$ naturally corresponding to $VT$. Define $R'\coloneqq\{\lk(v')\mid v'\in VT'\}$ similarly.
	Using Lemmas~\ref{lem:edge_vertex_int_admissible} and~\ref{lem:subgp_commensurability},  $R$ can be characterized as the set of maximal subsets $A$ of $ET$, such that for all $e,e'\in A$, the coarse intersection of $X_e$ and $X_{e'}$ is unbounded. Since quasi-isometries preserve having unbounded coarse intersection, we see $f$ induces a bijection $f_*:R\to R'$ taking $A$ to $\{f_*(a)\mid a\in A\}$, which is an element of $R'$. It follows that $f_*$ naturally induces a tree isomorphism $T\to T'$.
	It remains to show for every $v\in VT$, $d_\Haus(f(X_v),X_{f_*(v)})\leq B$. This follows immediately from the fact that $X_v$ has finite Hausdorff distance from $\bigcup_{e\in \lk v} X_e$.
\end{proof}

The fact that distinct vertex or edge spaces of $G$ are at infinite Hausdorff distance yields the following corollary.

\begin{cor}\label{cor:functionalmaner}
	Let $(X,T)$ be the tree of spaces associated to an admissible group. The following are satisfied for all quasi-isometries $f, g \colon X \to X$:
	\begin{enumerate}
		\item If $f$ and $g$ are close, then $f_{*} = g_{*}$.
		\item $(g \circ f)_{*} =  g_{*} \circ f_{*}$;
		\item $(id_{X})_{*} =  id_{T}$.
	\end{enumerate}
\end{cor}

\subsection{Admissible groups are quasi-isometrically rigid}
\begin{thm}\label{thm:admQIrigidity}
	Let $G$ be an admissible group. If $H$ is a finitely generated group quasi-isometric to $G$, then $H$ has a finite index subgroup which is also an admissible group.
\end{thm}

\begin{proof}
	Let $(X, T)$ be the tree of spaces associated with an admissible group $G$. Since $H$ is quasi-isometric to $G$, $H$ is also quasi-isometric to $X$. Thus $H$ admits a proper and cobounded quasi-action on $X$. It follows from Corollary~\ref{cor:functionalmaner} that the quasi-action of $H$ on $X$ induces an action of $H$ of $T$. Specifically if $\{f_h\}_{h\in H}$ is a quasi-action of $H$ on $X$, then $h\mapsto (f_h)_*$ is an action of $H$ on $T$.
	The action of $H$ on $T$ may have edge inversions, but passing to a subgroup $H^{\prime} \leq H$ with index at most 2 gives an  action of $H^{\prime}$ on $T$  without inversions.

	Suppose the quasi-action $H^{\prime} \curvearrowright_{q . a} X$ is a $(K,A)$-quasi-action. By enlarging $A$ if necessary, we can assume that for all $x,x'\in X$, there is some $h\in H'$ such that $d(x',f_h (x))\leq A$, and that $d_\Haus(f_h(X_x),X_{({f_h})_*(x)})\leq A$ for all $x\in VT\sqcup ET$ and $h\in H'$. Through a routine argument, it can be shown that the quotient $H'\backslash T$  is a finite graph, and the stabilizer $\operatorname{Stab}_{H'}(x)=H'_x$ of some $x \in XT\sqcup ET$ is quasi-isometric to the vertex or  edge space $X_x$. Thus the action of $H'$ on $T$ yields a finite graph of groups decomposition $\cG'$ of $H'$  where:
	\begin{enumerate}
		\item\label{proofThm1.1:1} The underlying graph $\Gamma'$ is the quotient $H'\backslash T$.

		\item\label{proofThm1.1:2}
		      Each vertex (resp.\ edge) group $H'_{x}$ of $\cG'$ is isomorphic to the stabilizer $H'_{\tilde x}$ of some vertex (resp.\ edge) $\tilde x$ of $T$ projecting to $x$ under the quotient  $T \to H'\backslash T$.
	\end{enumerate}

	We will show that $H'$ is an admissible group.
	According to~\cite{Gro81,pansu83croissance}, any finitely generated group that is quasi-isometric to $\Z^d$  contains a finite-index subgroup is isomorphic to $\Z^d$. By~(\ref{proofThm1.1:2}), every edge group of $\cG'$  is virtually $\mathbb{Z}^2$.  Applying  Theorem~A of~\cite{Mar22}, (\ref{proofThm1.1:2}) implies every vertex group of $\cG'$ is $\Z$-by-hyperbolic.  Thus $H'$ satisfies Conditions (1) and (2) of Definition~\ref{defn:extended}, with all vertex groups of type $\cS$.

	We now use  the construction of the action $H'\curvearrowright T$ and Lemmas~\ref{lem:edge_vertex_int_admissible} and~\ref{lem:subgp_commensurability} to deduce Conditions (3) and (4) of Definition~\ref{defn:extended}. It follows that for two distinct edges $e\neq e'\in T$, the edge spaces $X_e$ and $X_{e'}$ have bounded coarse intersection, so $H_e$ and $H_{e'}$ are not commensurable. This implies Condition (3) of Definition~\ref{defn:extended}. Moreover, if $e$ is an edge with $v=e_-$ and $w=e_+$. Let $e_v\in \lk(v)\setminus \{e\}$ and $e_w\in \lk(w)\setminus \{e\}$. Then by the above lemmas, we see that  $H'_{e}\cap H'_{e_v}$  and $H'_{e}\cap H'_{e_w}$  are two-ended subgroups with finite intersection, hence generate a finite index subgroup of $H'_e$ as it is virtually $\Z^2$.  Thus $H'$ satisfies  Condition (4) of Definition~\ref{defn:extended}, hence is an admissible group.
\end{proof}

\section{Quasi-isometric rigidity of extended admissible groups}

In this section, we are going to prove Theorems~\ref{thm:CKrigidity_qi} and~\ref{thm:CKrigidity_graphofgroups}.

\subsection{Bowditch boundary}
There are multiple equivalent definitions of relative hyperbolicity. We use relative hyperbolicity from~\cite{dructusapir2005treegraded} in previous sections and from~\cite{GM08} in this section.
Given a finitely generated group $G$ and a finite collection of finitely generated subgroups $\PP$, we say $S$ is a \emph{compatible} generating set of $(G,\PP)$ if $S$ generates $G$ and  $S\cap P$ generates $P$ for every $\PP$.

\begin{defn}[Combinatorial horoball~\cite{GM08}]
	Let $T$ be any graph with the vertex set $V$. We define the \emph{combinatorial horoball} based at $T$, $\mathcal{H}(=\mathcal{H}(T)$) to be the following graph:
		\begin{enumerate}
			\item $\mathcal{H}^{(0)}= V\times (\{0\}\cup \mathbb N)$.
			\item $\mathcal{H}^{(1)} = \{((t, n), (t, n + 1))\}\cup \set{((t_1, n), (t_2, n))}{d_T(t_1,t_2)\leq 2^n}$.
		\end{enumerate}
		We call edges of the first set \emph{vertical} and of the second \emph{horizontal}. The \emph{depth zero subgraph} of $\cH(T)$ is the full subgraph of $\cH(T)$ with vertex set $V\times \{0\}$.
\end{defn}
We note that the depth zero subgraph of $\cH(T)$ is isomorphic to $T$.

\begin{defn}[Cusped space~\cite{GM08}]\label{cspaces}
	Let $G$ be a finitely generated group and $\PP$ a finite collection of finitely generated subgroups of $G$. Let $S$ be a compatible finite generating set of $(G,\PP)$  and let $\Gamma(G,S)$ be the Cayley graph of $G$ with respect to $S$. For each left coset $gP$ of a subgroup $P\in\PP$, let $\mathcal{H}(gP)$ be the horoball based at  $T_{gP}$, the full subgraph of $\Gamma(G,S)$ with vertex set $gP$. The \emph{cusped space} $\Cusp(G,\PP,S)$ is the union of $\Gamma(G,S)$ with $\mathcal{H}(gP)$ for every left coset of $P\in \PP$, identifying the subgraph $T_{gP}$ with the depth zero subgraph of $\mathcal{H}(gP)$. We suppress mention of $S$ and $\PP$ when they are clear from the context.
\end{defn}

\begin{defn}[Relatively hyperbolic group~\cite{GM08}]\label{minimal}
	Let $G$ be a finitely generated group and $\PP$ a finite collection of finitely generated subgroups of $G$. Let $S$ be a compatible finite generating set of $(G,\PP)$. If the cusped space $\Cusp(G,\PP,S)$ is $\delta$-hyperbolic, then we say that $G$ is \emph{hyperbolic relative to} $\PP$ or that $(G,\PP)$ is \emph{relatively hyperbolic}. %
\end{defn}

\begin{defn}[Bowditch boundary~\cite{bowditch2012relatively}]\label{defn001}
	Let $(G,\PP)$ be a finitely generated relatively hyperbolic group. Let $S$ be a compatible finite generating set of $(G,\PP)$. The \emph{Bowditch boundary}, denoted $\partial(G,\PP)$, is the Gromov boundary of the associated cusped space, $\Cusp(G,\PP, S)$.
\end{defn}

\begin{rem}
	If $S$ and $S'$ are two compatible finite generating sets of $(G,\PP)$, there is a quasi-isometry $\Cusp(G,\PP,S)\to \Cusp(G,\PP,S')$. Consequently,  the notion  of a relatively hyperbolic group and its Bowditch boundary does not depend on the choice of finite compatible generating set. For convenience, we frequently suppress the choice of generating set and write $\Cusp(G,\PP)$ to denote $\Cusp(G,\PP,S)$ for some compatible finite generating set $S$.
\end{rem}

\begin{defn}
	The \emph{limit set} $\Lambda H$ of a subgroup $H \le G$ is defined to be the set of limit
	points of any $H$-orbit in the Bowditch boundary $\partial (G, \mathbb P)$.
	Suppose $\Lambda H$ is a subset with
	at least two points. The \emph{join} of $\Lambda H$, denoted $\join (\Lambda H)$, is the union of all geodesic lines
	in $\Cusp (G, \mathbb P)$ joining pairs of points in $\Lambda H$.
\end{defn}

\subsection{Relative hyperbolicity of extended admissible groups}
\begin{lem}\label{lem:NRHGadmissible}
	Admissible groups are not relatively hyperbolic groups.
\end{lem}
\begin{proof}
	Let $G$ be an admissible group. By Corollary~\ref{cor:ve_spaces_qi_embedded}, the inclusion of a vertex group $G_v \to G$ is a quasi-isometric embedding, and hence for any two points $x, y \in G_v$, a geodesic $\gamma$ in $G_v$ connecting $x$ to $y$ will be a uniform quasi-geodesic in $G$. This shows that the graph $G_v$ satisfies the quasi-convexity property as defined in~\cite[\S 4.1]{BD14}.
	Since every asymptotic cone of a vertex group of $G$ is without cut-points, it follows that vertex groups of $G$ are strongly algebraically thick of order zero in the sense of~\cite{BD14}. We have that $G$ is strongly thick of order at most $1$ since graphs of groups with infinite edge groups and whose vertex groups are thick of order $n$, are thick of order at most $n+1$, by~\cite[Proposition~4.4 \& Definition 4.14]{BD14}. Thus $G$ is not a relatively hyperbolic group by~\cite[Corollary~7.9]{BDM09}. %
\end{proof}

Let $G$ be an extended admissible group with graph of groups $\cG$ and underlying graph  $\Gamma$. By the normal form theorem, for each connected subgraph $\Gamma'$ of $\Gamma$, there is a subgroup $G_{\Gamma'}\leq G$ which is the fundamental group of the graph of groups with underlying graph $\Gamma'$, and with vertex, edge groups and edge monomorphisms coming from $\cG$.
Let $\Lambda$ be the full subgraph of $\Gamma$ with vertex set $\{v\in V\Gamma\mid \cG_v \text{ is type $\cS$}\}$. For each component $\Gamma'$ of $\Lambda$, we say that $G_{\Gamma'}$ is \begin{enumerate}
	\item a \emph{maximal admissible component} if $\Gamma'$ contains an edge;
	\item an \emph{isolated type $\cS$ vertex group} if  $\Gamma'$ consists  of a single vertex of type $\cS$.
\end{enumerate}  is a subgroup $G_{\Gamma'}\leq G$ for some connected component $\Gamma'$ of $\Lambda$.

The next lemma can be deduced from the Combination Theorem of relatively hyperbolic groups~\cite[Theorem~0.1]{Dah03}.

\begin{lem}\label{lem:RHGadmissible}
	Let $G$ be an extended admissible group with the graph of groups structure $\mathcal{G}$ such that it contains at least one vertex group of type $\mathcal{H}$.
	Let $G_1, . . . , G_k$ be the maximal admissible components
	and isolated vertex pieces of type $\mathcal{S}$ of an extended admissible group $G$. Let $G_{e_1}, \ldots, G_{e_m}$ be the edge groups so that both its associated vertex groups $G_{(e_i)_{\pm}}$ are of type $\mathcal{H}$, and let $T_1, \ldots, T_{\ell}$ be groups in $\cup \PP_v$ which are not edge groups of $G$.
	Then $G$ is hyperbolic relative to
	\[
		\PP = \{G_i\}_{i=1}^{k} \cup \{G_{e_s}\}_{s =1}^{m} \cup \{T_i\}_{i=1}^{\ell}
	\]
\end{lem}

\begin{rem}\label{rem:anothersplitting}
	Note that $G$ has another graph of groups decomposition $\mathcal{C}$ (not the same as $\mathcal{G}$). In $\cC$, vertex groups are either maximal admissible components of $G$, isolated vertex groups  of type $\mathcal{S}$, or vertex groups of type $\mathcal{H}$. Edge groups are virtually $\Z^2$
\end{rem}

\subsection{JSJ tree for splittings}
By a \emph{splitting} of a group, we mean a realization of $G$ as the fundamental
group of a reduced finite graph of groups with at least one edge. A splitting is said to be over a class $\mathcal{E}$ if every
edge group in the splitting is an element of $\mathcal{E}$. A splitting of a group $G$ is said to
be \emph{relative} to a collection of subgroups $\mathbb{P}$ if every $P \in \mathbb{P}$ fixes a vertex of the Bass--Serre tree.

If $(G,\mathbb{P})$ is relatively hyperbolic, a \emph{cut point} of the Bowditch boundary $\partial(G, \mathbb P)$  is a point $\xi \in \partial(G, \mathbb P)$ such that $\partial(G, \mathbb P) \setminus \{\xi\}$ is disconnected.
A subset $C$ in $\partial(G, \mathbb P)$  is a \emph{cyclic element} if $C$ consists of a single
cut point or contains a non-cutpoint $p$ and all points $q$ that are not separated
from $p$ by any cut point of $\partial (G, \mathbb P)$. A cyclic element is \emph{non-trivial} if it contains at
least two points.

The following theorem is cited from~\cite[Theorem~8.1]{HH23}, in which it is referred to as a combination of~\cite[Theorem~9.2]{Bow01} and~\cite[Theorem~1.1]{DH23}.

\begin{thm} [{\cite{Bow01,DH23}}]\label{thm:HH238.1}
	Let $(G, \mathbb{P})$ be relatively hyperbolic with connected boundary $M=\partial(G, \mathbb{P})$. Let $T$ be the bipartite graph with vertex set $\mathcal{V} \sqcup \mathcal{W}$, where $\mathcal{V}$ is the set of cut points and $\mathcal{W}$ is the set of non-trivial cyclic elements of $M$. Two vertices $v \in \mathcal{V}$ and $w \in \mathcal{W}$ are connected by an edge in $T$ if and only if the cut point $v$ is contained in the cyclic element $w$.

	Then the graph $T$ is a JSJ tree for splittings of $G$ over parabolic subgroups relative to $\mathbb{P}$. There are only finitely many $G$-orbits of edges of $T$,   and the stabilizer of each edge is finitely generated.
\end{thm}

The following result is a combination of~\cite[Section~7]{Bow01} and~\cite[Theorem~1.3]{Bow01}.
\begin{prop}[{\cite[Proposition~8.2]{HH23}}]\label{prop:Prop8.2GH23}
	Let $(G, \mathbb P)$ be relatively hyperbolic with connected boundary, and let $C$ be a non-trivial cyclic element of $\partial (G, \mathbb P)$. Then the following hold:
	\begin{enumerate}
		\item The set $C$ is connected and locally connected.

		\item The stabilizer $H$ of $C$ is hyperbolic relative to a family $\mathcal{O}$ of representatives of the conjugacy classes of infinite subgroups of the form $H \cap gPg^{-1}$
		      where $g \in G$ and $P \in \mathbb P$. Additionally, the boundary $\partial (H, \mathcal{O})$ is $H$-equivariantly homeomorphic to $C$.
	\end{enumerate}
\end{prop}

\begin{rem}\label{rem:samesplittings}
	In Theorem~\ref{thm:HH238.1}, it is shown that $T$ is the  JSJ tree $T$ for splittings of $G$ over parabolic subgroups relative to $\mathbb P$. This tree is referred to as the \emph{maximal peripheral splitting} in~\cite{Bow01}. For more details, see~\cite[Theorem~9.2]{Bow01}.

	In the setting of
	Lemma~\ref{lem:RHGadmissible}, recall that no vertex group of type $\mathcal{H}$ splits over any subgroup of a peripheral subgroup. This guarantees that vertex groups of type $\mathcal{H}$ are elliptic in any peripheral splitting. Therefore, the splitting $\mathcal{C}$ in Remark~\ref{rem:anothersplitting} must be a maximal peripheral splitting in the sense of Bowditch~\cite{Bow01}, and consequently from the above paragraph, the JSJ tree $T$ for splitting of $G$ over parabolic subgroups relative to $\mathbb P$ constructed in Theorem~\ref{thm:HH238.1}  coincides with the splitting $\mathcal{C}$ of $G$.
	Combined with part~(2) of Proposition~\ref{prop:Prop8.2GH23}, it follows that if $G_v$ is a vertex group of type $\mathcal{H}$ in the graph of groups structure $(G, \mathcal{C})$, it is the stabilizer of a non-trivial cyclic element $C$ in the JSJ tree $T$ for splitting of $G$ over parabolic subgroups relative to $\mathbb P$.
\end{rem}

\subsection{Proofs of Theorem~\ref{thm:CKrigidity_graphofgroups} and Theorem~\ref{thm:CKrigidity_qi} }
We make use of the following result, which is a special case of a theorem of Groff~\cite{groff2013quasiisometries}. We note that $G$ can be identified with a subset of vertices of a cusped space $\Cusp(G,\mathbb P)$.
\begin{prop}[{\cite[Theorem 6.3]{groff2013quasiisometries}}]\label{prop:inducedqi}
	Let $G$ be a  finitely generated group and let $f:G\to G$ be a $(K,A)$-quasi-isometry. Suppose $G$ is hyperbolic relative to $\mathbb P$, and that no $P\in \mathbb P$ is relatively hyperbolic. Then  $f$ extends to a $(K_1,A_1)$-quasi-isometry $F:\Cusp(G,\mathbb P)\to \Cusp(G,\mathbb P)$, where $K_1$ and $A_1$ depend only on $G$,  $K$ and $A$.
\end{prop}
\begin{rem}
	Although the dependence of $K_1,A_1$ only on $G$, $K$ and $A$ is not evident in the statement of {\cite[Theorem 6.3]{groff2013quasiisometries}}, it is shown in the  proof of this theorem.
	We also note that although parts of~\cite{groff2013quasiisometries} are incorrect, the proof of the preceding result is correct; see~\cite{HH23}. 
\end{rem}

The following theorem easily implies Theorem~\ref{thm:CKrigidity_graphofgroups} from the introduction.
If $G$ is  an extended admissible group with associated tree of spaces $(X,T)$, we use the terms
``type $\mathcal{S}$'' and ``type $\mathcal{H}$'' to describe the associated vertex spaces or left cosets associated to vertex groups of type $\cS$ or $\cH$ as in Definition~\ref{defn:extended}.

\begin{thm}\label{thm:vertextovertexextend}
	Let $G$ be an extended admissible group and let $(X,T)$ be the associated tree of spaces.
	For every $(K,A)$-quasi-isometry $f:X\to X$ and vertex space $X_v$, there exists a unique vertex space $X_{v'}$ of the same type  such that the Hausdorff distance between $f(X_v)$ and $X_{v'}$ is at most $ B=B(K,A,X)$.
\end{thm}
\begin{proof}
	Let $\cG$ be the graph of groups associated to $G$.
	By Proposition~\ref{prop:treeofspacesBST}, it is enough to prove the analogous statement with  quasi-isometries $f:G\to G$ and left cosets of vertex groups of $\cG$ instead of vertex spaces.
	By Lemma~\ref{lem:NRHGadmissible}, admissible groups are not relatively hyperbolic. Hence by Lemma~\ref{lem:RHGadmissible}, $G$ is hyperbolic relative to a collection $\PP$ of non-relatively hyperbolic subgroups. 
	
	Now suppose $G_v$ is a  type $\cH$ vertex group of $G$.  As peripheral subgroups are infinite and $G$ does not split over a finite subgroup relative to $\mathbb P$, the Bowditch boundary $\partial (G, \mathbb P)$ is connected.
	By Proposition~\ref{prop:Prop8.2GH23} and Remark~\ref{rem:samesplittings}, $G_v$ stabilizes a non-trivial cyclic element $C_v$ of $\partial(G,\PP)$, hence is hyperbolic relative to a family $\mathcal O_v$ as in Part~(2) of Proposition~\ref{prop:Prop8.2GH23}, and $\partial (G_v,\mathcal O_v)$ is $G_v$-equivariantly homeomorphic to $C_v$. Moreover, as $G_v$ is a vertex group in the maximal peripheral splitting of $G$,~\cite[Lemma~6.1]{HH23} ensures $G_v$ is relatively quasiconvex in $(G, \mathbb P)$. Thus by~\cite[Proposition~7.6]{Hru10},  the action of $G_v$ on $\join (\Lambda G_v)$   is cusp uniform, and hence the action of $G_v$  on the truncated space \[\widetilde{\join} (\Lambda G_v) : = \join (\Lambda G_v) \cap G\]  is cocompact, where $G$ is identified with a subset of vertices of $\Cusp(G,\mathbb P)$. This implies that $G_v \subseteq \Cusp (G, \mathbb P)$    has finite Hausdorff distance from $\widetilde{\join} (\Lambda G_v)$
	since they are both $G_v$-invariant. Since there are finitely many vertex groups, there is a constant $A_1=A_1(G,\PP)$ such that
	\begin{align*}
		d_{\Haus} (G_v, \widetilde{\join}(\Lambda (G_v))) \le A_1
	\end{align*}
	for all type $\cH$ vertex groups $G_v$. Hence for each $g\in G$ and type $\cH$ vertex group $G_v$ of $G$, we have \begin{align} \label{eqn:joinlimit2}
		d_{\Haus} (gG_v, \widetilde{\join}(g\Lambda (G_v))) \le A_1.
	\end{align}

	By Proposition~\ref{prop:inducedqi},  $f$ induces a quasi-isometry  of cusped spaces $F: \Cusp(G, \mathbb P) \to \Cusp(G, \mathbb P)$, and hence induces a homeomorphism $\partial {F} \colon \partial  (G, \mathbb P) \to \partial  (G, \mathbb P)$ of the Bowditch boundary. Moreover, the quasi-isometry constants of $F$ depend only on $G$, $K$ and $A$.
	Since  $\partial {F}$ preserves non-trivial cyclic elements,  Proposition~\ref{prop:Prop8.2GH23} and Remark~\ref{rem:samesplittings} ensures that  there is a type $\cH$ vertex group  $G_w$ and $g\in G$ such that
	\[
		\partial F (\Lambda (G_v)) = \Lambda (gG_wg^{-1})=g\Lambda (G_w).
	\]
	The Extended Morse Lemma for $\delta$-hyperbolic spaces implies that
	there is a constant $A_2= A_2 (K, A, G)$ such that
	\[
		d_{\Haus} (F (\join (\Lambda G_v)), g\join (\Lambda G_w)) \le A_2.
	\]
	Since $F$ is an extension of $f$, it follows that 
	\[
		d_{\Haus} (f (\widetilde {\join} (\Lambda G_v)), g \widetilde {\join} (\Lambda G_w)) \le A_3.
	\]
	for some  $A_3 = A_3 (K, A, G)$. Combining this with (\ref{eqn:joinlimit2}) yields
	\[
		d_{\Haus}(f(G_v),g G_w) \leq A_4
	\] for some $A_4 = A_4 (K, A, G)$, as required. The uniqueness of the left coset $gG_w$ follows from Proposition~\ref{prop:treeofspacesBST} and Proposition~\ref{prop:coarse_int_vertex_edge}.
\end{proof}

\begin{cor}\label{cor:tree_isom_extend}
	Let $G$ be an extended admissible group and let $(X,T)$ be the associated tree of spaces.
	For any $(K,A)$-quasi-isometry $f:X\to X$, there is a constant $B=B(K,A,X)$ such that the following holds. There is a unique tree isomorphism $f_*:T\to T$ such that
	\[d_\Haus(f(X_x),X'_{f_*(x)})\leq B
	\]
	for every $x\in VT\cup ET$.
\end{cor}
\begin{proof}
	By Theorem~\ref{thm:vertextovertexextend}, $f$ induces  a bijection $f_*:VT\to VT$ such that $d_\Haus(f(X_v),X_{f_*(v)})\leq B$ for some $B=B(K,A,X)$. By Lemmas~\ref{lem:edge_vertex_int_admissible} and~\ref{lem:subgp_commensurability}, the coarse intersection of two vertex spaces $X_v$ and $X_w$ is quasi-isometric to $\mathbb Z^2$ if and only if $v$ and $w$ are adjacent. Since quasi-isometries preserve coarse intersection of subspaces,  $v$ and $w$ are adjacent if and only if $f_*(v)$ and $f_*(w)$ are. Thus $f_*$ induces a unique tree isomorphism, which we also call $f_*$.
\end{proof}

We now prove Theorem~\ref{thm:CKrigidity_qi}.
\begin{proof}[Proof of Theorem~\ref{thm:CKrigidity_qi} and Corollary~\ref{cor:qi_class}]
	Assume $G$ is an extended admissible group with the graph of groups structure $\mathcal{G}$ and tree of spaces $(X,T)$. If $G$ does not have a vertex group of type $\mathcal{H}$, then $G$ is an admissible group, and we apply Theorem~\ref{thm:admQIrigidity}. We thus  assume that $G$ has at least one vertex group of type $\mathcal{H}$.
	Theorem~\ref{thm:CKrigidity_graphofgroups} implies that there exists a tree isomorphism $f_*:T\to T$, such that  $f(X_x)$ is at  uniform finite Hausdorff distance from $X_{f_*(x)}$  for every $x\in VT\cup ET$.
	
	Let $G'$ be a finitely generated group quasi-isometric to $G$. Following the proof  of Theorem~\ref{thm:admQIrigidity}, using Corollary~\ref{cor:tree_isom_extend} instead of Theorem~\ref{thm:qipreservespieces},  we see after replacing  $G'$  by a subgroup of index at most two,  $G'$ has a graph of groups structure $\cG'$, where the edge groups of $\cG'$ are virtually $\Z^2$, and vertex groups of $\cG'$ are quasi-isometric to vertex groups of $\cG$. According to~\cite[Theorem~A]{Mar22}, if a group is quasi-isometric to a vertex group of type $\mathcal{S}$, then that group is $\Z$-by-hyperbolic. Groups that are quasi-isometric to relatively hyperbolic groups with virtually $\Z^2$ peripheral subgroups are also relatively hyperbolic groups with virtually $\Z^2$ peripheral subgroups, as shown in~\cite{dructusapir2005treegraded} and~\cite{BDM09}. Moreover,~\cite[Theorem~1.1]{DH23} states that for a relatively hyperbolic group $(G, \PP)$, the existence of a nontrivial splitting relative to $\PP$ is equivalent to the existence of cut-points on its Bowditch boundary. Since Bowditch boundaries are quasi-isometric invariant, it follows that groups quasi-isometric to groups of type $\mathcal{H}$ are also groups of type $\mathcal{H}$. Therefore, vertex groups of $\cG'$ are either of type $\mathcal{S}$ or type $\mathcal{H}$. The graph of groups  $\cG'$ satisfies conditions (1) and (2) of Definition~\ref{defn:extended}. Conditions (3) and (4) are also satisfied, using an identical argument to that given in Theorem~\ref{thm:admQIrigidity}.
	Therefore, we can conclude that $G'$ is also an extended admissible group.
	
	To deduce Corollary~\ref{cor:qi_class}, all that remains is to show that if two type $\cS$ vertex groups are quasi-isometric, then their hyperbolic quotients are quasi-isometric. This follows from a result of Kapovich--Kleiner--Leeb~\cite{kapovichkleinerleeb1998quasiisometries}, who show that any quasi-isometry between $\mathbb Z$-by-hyperbolic groups induces a quasi-isometry between their hyperbolic quotients.
\end{proof}

In the rest of this paper, we will  prove Corollary~\ref{cor:comm}.

\begin{defn}
	Let $G$ be a finitely generated group acting geometrically on a proper
	geodesic hyperbolic space $X$. A \emph{$G$-symmetric pattern} $\mathcal{J}$ in $X$ is a non-empty $G$-invariant collection
	of quasi-convex subsets of $X$ such that:
	\begin{enumerate}
		\item for every $J \in \mathcal{J}$, the stabilizer $\Stab_{G}(J)$ acts cocompactly on $J$ and is an infinite, infinite-index subgroup of $G$;
		\item $\mathcal{J}$  contains only finitely many $G$-orbits.
	\end{enumerate}
\end{defn}

A \emph{symmetric pattern in $X$} is a $G$-symmetric pattern for some finitely generated group $G$ acting geometrically on $X$. We denote $X$ together with a symmetric pattern $\mathcal{J}$ by $(X, \mathcal{J})$. A \emph{pattern-preserving quasi-isometry} $f:(X, \mathcal{J}) \rightarrow\left(X^{\prime}, \mathcal{J}^{\prime}\right)$ is a quasi-isometry $f: X \rightarrow X^{\prime}$ such that there exists a constant $A \geq 0$ so that:
\begin{enumerate}
	\item for all $J_1 \in \mathcal{J}$, there exists a $J_2 \in \mathcal{J}^{\prime}$ such that $d_{\Haus}\left(f(J_1), J_2\right) \leq A$;
	\item for all $J_2 \in \mathcal{J}^{\prime}$, there exists a $J_1 \in \mathcal{J}$ such that $d_{\Haus}\left(f(J_1), J_2\right) \leq A$.
\end{enumerate}

Let $\text{QI}(X, \mathcal{J})\leq\text{QI}(X)$ be the subgroup of  equivalence classes of pattern-preserving quasi-isometries of $\mathcal{J}$.

\begin{thm}[\cite{Bis12}]\label{thm:Biwas}
	Suppose $\mathcal{J}$ and $\mathcal{J}^{\prime}$ are symmetric patterns in $\mathbb{H}^n$ for some $n \geq 3$. If $f:\left(\mathbb{H}^n, \mathcal{J}\right) \rightarrow\left(\mathbb{H}^n, \mathcal{J}^{\prime}\right)$ is a pattern-preserving quasi-isometry, then there is a hyperbolic isometry $f^{\prime}: \mathbb{H}^n \rightarrow \mathbb{H}^n$ such that \[\sup _{x \in \mathbb{H}^n} d\left(f(x), f^{\prime}(x)\right)<\infty.\]
\end{thm}
Theorem~\ref{thm:Biwas} has the following corollary; see~\cite{Bis12,MSSW23}.

\begin{cor}

	Suppose a group $G$ acts faithfully, discretely  and cocompactly   on $\mathbb{H}^n$ for some $n \geq 3$ and $\mathcal{J}$ is a $G$-symmetric pattern in $\mathbb{H}^n$. Then, $\mathrm{QI}\left(\mathbb{H}^n, \mathcal{J}\right)$ can be identified with a discrete subgroup of $\operatorname{Isom}\left(\mathbb{H}^n\right)$, and under this identification, $G$ is a finite-index subgroup of $\mathrm{QI}\left(\mathbb{H}^n, \mathcal{J}\right)$.
\end{cor}

Let $G$ be an admissible group with associated graph of groups $\cG$. For each vertex group $G_v$ with hyperbolic quotient $Q_v$, we note that $(Y_v,\{\ell_e\})$ is a $Q_v$-symmetric pattern, where  $(Y_v,\{\ell_e\})$ is as in  Section~\ref{sub:geovertexedge}.

\begin{lem}\label{lem:finiteindexinrelativeQI}
	Let $G$ be an admissible group with associated graph of groups $\cG$. Suppose every hyperbolic quotient $Q_v$ of a vertex group of $\cG$  is the fundamental group of a closed hyperbolic $n_v$--manifold with $n_v \ge 3$. Then each  hyperbolic quotient $Q_{v}$ of $\cG$ is a finite index subgroup of  $\mathrm{QI}(Y_v, \{\ell_e \})$.
\end{lem}

\begin{proof}
		Using the Milnor--Schwarz Lemma, there is a $Q_{v}$-equivariant quasi-isometry $h_v \colon Y_v \to \mathbb{H}^{n_v}$. The quasi-inverse of $h_v$ is denoted by $\bar{h}_v$.
	Since the collection $\{\ell_e\}$ is a $Q_{v}$-symmetric pattern of $Y_v$, it maps to  a $Q_{ v}$-symmetric pattern, $\mathcal{J}_v$, of $\mathbb{H}^{n_v}$ under $h_v$. It follows that $\mathrm{QI}(Y_v, \mathcal{L}_v)$ and $\mathrm{QI}(\mathbb{H}^{n_v}, \mathcal{J}_v)$, are isomorphic, since $h_v$ is a pattern-preserving quasi-isometry.
	By Theorem~\ref{thm:Biwas}, $Q_{ v}$ is a finite index subgroup of $\mathrm{QI}(\mathbb{H}^{n_v}, \mathcal{J}_v)$ and is also a finite index subgroup of $\mathrm{QI}(Y_v, \mathcal{L}_v)$. This proves the lemma.
\end{proof}

\begin{proof}[Proof of Corollary~\ref{cor:comm}]
	Let $\cG$ be the graph of groups associated to $G$.
	According to Theorem~\ref{thm:CKrigidity_qi}, the finitely generated group $G'$ splits as a graph of groups $\mathcal{G}'$ whose edge groups are virtually $\mathbb{Z}^2$ and  whose  vertex groups are quasi-isometric to those of $G$. By the proof of Corollary~\ref{cor:qi_class}, it follows that for every vertex group $w$ of $\cG'$, there is a vertex group $v$ of $\cG$ and a pattern preserving quasi-isometry $(Y_v, \{\ell_e \})\to(Y'_w, \{\ell'_e \})$. Thus $\mathrm{QI}(Y_v, \{\ell_e\})$ is isomorphic to $\mathrm{QI}(Y_w, \{\ell_e'\})$ and Lemma~\ref{lem:finiteindexinrelativeQI} implies that modulo finite normal subgroups, the hyperbolic quotients $Q_v$ and $Q_w$ are finite-index subgroups of $\mathrm{QI}(Y_v, \{\ell_e\})$. Thus $Q_v$ and $Q_w$ are virtually isomorphic.
\end{proof}

\bibliographystyle{alpha}
\bibliography{Alex_Hoang}
\end{document}